%% file: main.tex
\documentclass[smallextended]{svjour3}       
\smartqed  
\usepackage{graphicx}
\usepackage{amsmath, amssymb}
\usepackage[dvipsnames]{color}
\usepackage{multirow}
\usepackage[justification=centering]{caption}

\newtheorem{coro}{Corollary}

\include{symbols}

\begin{document}

\title{Robust block preconditioners for poroelasticity\thanks{
This work of Jinchao Xu is supported in part by the U.S. Department of Energy, Office of Science, Office of Advanced Scientific Computing Research as part of the Collaboratory on Mathematics for Mesoscopic Modeling of Materials under contract number DE-SC0009249.
}
}
\subtitle{
}


\author{Shuangshuang Chen \and Qingguo Hong \and Jinchao Xu$^{*}$ \and 
Kai Yang         
}


\institute{Shuangshuang  Chen\at      
Beijing Institute for Scientific and Engineering Computing (BISEC),\\
 Beijing University of Technology, Beijing, 100124, China.\\
 \email{chenshuangshuang@bjut.edu.cn}
\and
           Qingguo Hong \at
             The Center for Computational Mathematics and Applications,\\
              Department of Mathematics, \\Pennsylvania State University, University Park, PA 16802, U.S.A\\
               \email{huq11@psu.edu} 
           \and
$*$ Corresponding author: Jinchao Xu \at
              The Center for Computational Mathematics and Applications,\\
              Department of Mathematics, \\Pennsylvania State University, University Park, PA 16802, U.S.A\\
              \email{xu@math.psu.edu}         
\and
 Kai Yang \at
            The Center for Computational Mathematics and Applications,\\
              Department of Mathematics, \\Pennsylvania State University, University Park, PA 16802, U.S.A\\
               \email{yangkai1001@gmail.com} 
}

\date{Received: date / Accepted: date}

\maketitle
\begin{abstract}
In this paper we study the linear systems arising from discretized poroelasticity problems. We formulate one block preconditioner for the two-filed Biot model and several preconditioners 
for the classical three-filed Biot model under the unified relationship framework between well-posedness and preconditioners.  By the unified theory, we show all the considered preconditioners are uniformly optimal with respect to material and discretization parameters.  Numerical tests demonstrate the robustness of these preconditioners.
\end{abstract}

\section{Introduction}
Poroelasticity, the study of the fluid flow in porous and elastic media,  couples the elastic deformation with the fluid flow in porous media.  The Biot model has wide applications in geoscience, biomechanics, and many other fields.  Numerical simulations of poroelasticity are challenging. A notorious instability of the numerical discretization method for the poroelasticity model is the unphysical oscillation of the pressure under certain conditions \cite{Murad.M;Loula.A1994a}.  There are many possible sources of this instability. One of the most significant sources is the instability of the finite element approximation for the coupled systems \cite{Haga.J;Osnes.H;Langtangen.H2012b,Axelsson.O;Blaheta.R;Byczanski.P2012a}.  This motivates us to study the well-posedness of the finite element discretization. However, we do not look further into the details of the instability of the Biot model and refer interested readers to \cite{Haga.J;Osnes.H;Langtangen.H2012b,Axelsson.O;Blaheta.R;Byczanski.P2012a,Aguilar.G;Gaspar.F;Lisbona.F;Rodrigo.C2008a}.

Another challenge associated with the Biot model is that of developing efficient linear solvers.  Direct solvers have poor performance when the size of problems becomes large.  Iterative solvers are good alternatives, as they exhibit better scalability.  However, the convergence of iterative solvers is very much problem-dependent such that there is a need for robust preconditioners.  For example, the multigrid preconditioned Krylov subspace method usually has optimal convergence rate for the Poission equation and many other symmetric positive definite problems \cite{Hackbusch.W1985a,Xu.J1992a}.  However, for poroelasticity problems, coupled systems of equations must be solved, which are known to be indefinite and ill-conditioned \cite{Ferronato.M;Gambolati.G;Teatini.P2001a}.
  Preconditioning techniques for poroelasticity problems have been the subject of considerable research in the literature \cite{Axelsson.O;Blaheta.R;Byczanski.P2012a,Lipnikov.K2002a,Bergamaschi.L;Ferronato.M;Gambolati.G2007a,Phoon.K;Toh.K;Chan.S;Lee.F2002a,Toh.K;Phoon.K;Chan.S2004a,Haga.J;Osnes.H;Langtangen.H2011a,Haga.J;Osnes.H;Langtangen.H2012a,Turan.E;Arbenz.P2014a} and most of the techniques developed are based on the Schur complement approach. In \cite{Phoon.K;Toh.K;Chan.S;Lee.F2002a,Toh.K;Phoon.K;Chan.S2004a}, diagonal approximation of the Schur complement preconditioner is used to precondition two-field formulation of the Biot model.  In \cite{Haga.J;Osnes.H;Langtangen.H2011a,Haga.J;Osnes.H;Langtangen.H2012a}, Schur complement preconditioners are also studied for two-field formulation with the algebraic multigrid (AMG) as the preconditioner for the elasticity block.  In \cite{Axelsson.O;Blaheta.R;Byczanski.P2012a}, Schur complement approaches for three-field formulation are investigated. Recently, robust block diagonal and block triangular preconditioners are developed in \cite{adler2017robust}  for two-field Biot model. And for classical 
three-filed Biot model, the robust block preconditioners are designed in \cite{hong2018parameter,hong2018conservative} based on the uniform stability estimates. Robust preconditioner for a new three-field formulation introducing a total pressure as the third unknown is analyzed in~\cite{lee2017parameter}. Robust block diagonal and block triangular preconditioners are also developed in \cite{adler2019robust} based on the descretization proposed in \cite{rodrigo2018new}. Other robustness analysis for fixed-stress splitting method and Uzawa-type method for multiple-permeability poroelasticity systems are presented in \cite{hong2018fixed} and \cite{hong2019parameter}.

The focus of this paper is on the stability of the linear systems after time discretization and several robust preconditioners for the iterative solvers under the unified relationship framework between well-posedness and preconditioners. The block preconditioners in \cite{adler2017robust} for two field formulation and in \cite{hong2018parameter,adler2019robust} for the three field formulation can be briefly written in this framework.  In addition, we analyze the well-posedness of the linear systems and propose other optimal preconditioners for the Biot model \cite{Haga.J;Osnes.H;Langtangen.H2012b} based on the mapping property \cite{Mardal.K;Winther.R2011a}.  By proposing optimal block preconditioners, we convert the solution of complicated coupled system into that of a few SPD systems on each of the fields. 

The rest of this paper is organized as follows. In Section \ref{sec:biot}, we give a brief introduction of the Biot model. In Section \ref{sec:well-posed}, we introduce two theorems in order to prove well-posedness. In Section \ref{sec:precondition},  we address the unified framework indicating the relationship between preconditioning and well-posedness of linear systems.  In Section \ref{sec:2field} and Section \ref{sec:3field}, we show the well-posedness and several optimal preconditioners for the Biot model under the unified framework.  In Section \ref{sec:numerics}, we present numerical examples to demonstrate the robustness of these preconditioners. 

\section{The Biot model}
\label{sec:biot}
The poroelastic phenomenon is usually characterized by the Biot model \cite{Biot.M1941a,Biot.M1955a}, which couples structure displacement $\mathbf{u}$, fluid flux $\mathbf{v}$, and fluid pressure $p$. Consider a bounded and simply connected Lipschitz domain $\Omega\subset\mathbb{R}^n (n=2,3)$ of poroelastic material.  As the deformation is assumed to be small, we assume that the deformed configuration coincides with the undeformed reference configuration.  Let $\sigma$ denote the total stress in this material. From the balance of the forces, we first have
$$-\nabla\cdot\sigma=f,~~\mbox{ in }\Omega.$$
In addition to the elastic stress 
$$\sigma_e=2 \mu \epsilon(\mathbf{u}) + \lambda (\nabla \cdot  \mathbf{u}) \bI,$$ 
the fluid pressure also contributes to the total stress, which results in the following constitutive equation:
$$\sigma=\sigma_e-\alpha p\mathbf{I}.$$
Here, $\mu:=\frac{E}{2(1+\nu)}$ and $\lambda:=\frac{\nu E}{(1+\nu)(1-2\nu)}$ are the Lam\'e constants and $\nu\in[0,1/2)$ is the Poisson ratio, the symmetric gradient is defined by $\epsilon(\mathbf{u}):=(\nabla\mathbf{u}+\nabla\mathbf{u}^T)/2$, and  $\alpha$ is the Biot-Willis constant.  Therefore, we obtain the following momentum equation
$$- \nabla \cdot (2 \mu \epsilon(\mathbf{u}) + \lambda (\nabla \cdot  \mathbf{u}) \bI - \alpha p \bI) = f, ~~\mbox{ in }\Omega.$$

Let $\eta$ denote the fluid content. Then, the mass conservation of the fluid phase implies that

\begin{equation}
\label{eq:mass_cons}
\partial_t\eta+\nabla\cdot \bv=g~~\mbox{ in } \Omega,
\end{equation}
where $g$ is the source density.  The fluid content is assumed to satisfy the following constitutive equation:
\begin{equation}
\label{eq:constitutive_eta}
\eta = Sp+\alpha\nabla\cdot\bu,
\end{equation}
where $S$ is the fluid storage coefficient.  We also have the Biot-Willis constant $\alpha$ in this equation, as this poroelastic model is assumed to be a reversible process and the increment of work must be an exact differential \cite{Biot.M1941a,Rice.J2001a,CHENG.A2014a}.
Based on (\ref{eq:mass_cons}) and (\ref{eq:constitutive_eta}), the following equation holds
$$\alpha \nabla \cdot \dot{\mathbf{u}} + \nabla \cdot \mathbf{v} +S \dot{p}= g,~~\mbox{ in }\Omega.$$
According to Darcy's law, we have another equation:

$$k^{-1}\mathbf{v} +\nabla p =r,$$
where $k$ is the fluid mobility and $r$ is the body force for the fluid phase.

We consider all the parameters to be positive.  The following boundary conditions are assumed:
\begin{equation}
\label{eq:bdc_u}
\mathbf{u}=\mathbf{u}_D, \mbox{ on } \Gamma_{D,u},~~ \sigma \bn= g_N, \mbox{ on } \Gamma_{N,u},
\end{equation}
\begin{equation}
\label{eq:bdc_v}
\mathbf{v}\cdot \bn=\mathbf{v}_D,\mbox{ on } \Gamma_{D,v}, ~~ p = p_N, \mbox{ on } \Gamma_{N,v},
\end{equation}
where $\Gamma_{D,u}\cap\Gamma_{N,u}=\emptyset$, $\bar{\Gamma}_{D,u}\cup\bar{\Gamma}_{N,u}=\partial\Omega$ and $\Gamma_{D,v}\cap\Gamma_{N,v}=\emptyset$, $\bar{\Gamma}_{D,v}\cup\bar{\Gamma}_{N,v}=\partial\Omega$.

The initial conditions are as follows:
$$\mathbf{u}(x,0)=\mathbf{u}_0(x), ~~ p(x,0)=p_0(x), ~~ $$
where $\mathbf{u}_0$ and $p_0$ are given functions.

%
We use the backward Euler method to discretize the time derivative $\dot{\mathbf{u}}$:
$$\dot{\mathbf{u}}(t_n)\approx \frac{\mathbf{u}(t_n)-\mathbf{u}(t_{n-1})}{\Delta t},$$  
where $\Delta t$ is the time step size. More sophisticated implicit time discretizations result in similar linear systems.  As we are focusing on the properties of the linear systems resulting from the time discretized problem, we consider only the backward Euler method for the sake of brevity.  After the implicit time discretization, fast solvers are needed to solve the following three-field system:
\begin{equation}
\label{eq:biot_3field}
\left\{
\begin{aligned}
-\nabla\cdot(2 \mu \epsilon(\mathbf{u}) + \lambda (\nabla \cdot \mathbf{u})\bI)+\alpha\nabla p&=f,\\
k^{-1}\mathbf{v}+\nabla p&=r,\\
\frac{\alpha}{\Delta t}\nabla\cdot\mathbf{u}+\nabla\cdot\mathbf{v}+\frac{S}{\Delta t} p&=g.\\
\end{aligned}
\right.
\end{equation}
Note that the right-hand side of the last equation in (\ref{eq:biot_3field}) includes terms from previous time step due to the time discretization.  As the exact form of this right-hand side does not affect the well-posedness of the linear system, we keep using $g$ to denote it.  We apply this convention to all the right-hand sides in similar situations, throughout the rest of this paper.

To reduce the number of variables, the fluid flux $\bv$ is eliminated to obtain the following two-field system: 
\begin{equation}
\label{eq:2by2}
\left\{
\begin{aligned}
-\nabla\cdot(2 \mu \epsilon(\mathbf{u}) + \lambda (\nabla \cdot \mathbf{u})\bI)+\alpha\nabla p&=f,\\
\frac{\alpha}{\Delta t}\nabla\cdot\mathbf{u}-k\Delta p+\frac{S}{\Delta t}p&=g.\\
\end{aligned}
\right.
\end{equation}

In the rest of this paper, we develop block preconditioners for both the two-field and three-field systems.

\section{Well-posedness of linear systems}
\label{sec:well-posed}
In this section, we first introduce several theorems to prove the well-posedness of the following saddle point problem:
Find $(u,p)\in {\mathbb M}\times{\mathbb N}$ such that $\forall (\phi,q)\in {\mathbb M}\times {\mathbb N}$, the following equations hold
\begin{equation}
\label{eq:saddle_abc}
\left\{
\begin{aligned}
a(u,\phi)&+ b(\phi,p)&=& \langle f,\phi\rangle,\\
b(u,q)&-c(p,q)&=&\langle{g,q}\rangle.\\
\end{aligned}
\right.
\end{equation}
Here, $\mathbb{M}$ and $\mathbb{N}$ are given Hilbert spaces with the inner products $(\cdot,\cdot)_{\mathbb{M}}$ and $(\cdot,\cdot)_{\mathbb{N}}$, respectively. The corresponding norms are denoted by $\|\cdot\|_{\mathbb{M}}$ and $\|\cdot\|_{\mathbb{N}}$.

Given $b(\cdot,\cdot)$, the following kernel spaces are important in the analysis:
$${\mathbb Z}=\{u\in{\mathbb M} | b(u,q)=0, \forall q\in {\mathbb N}\},$$
$${\mathbb K} = \{p\in \mathbb{N} | b(\phi,p)=0,\forall \phi\in\mathbb{M}\}.$$

We consider the orthogonal decomposition of $u\in \mathbb{M}$ and $p\in \mathbb{N}$ as follows:
$$u=u_0+\bar{{u}},~~u_0\in \mathbb{Z},~ \bar{{u}}\in \mathbb{Z}^{\perp}, ~~~p=p_0+\bar p, ~~p_0\in \mathbb{K},~ \bar p\in\mathbb{K}^{\perp}.$$

We will use these notation to denote the components of functions in the kernel spaces and their orthogonal complements throughout the rest of this section. 

The well-posedness of (\ref{eq:saddle_abc}) can be proved provided that $a(\cdot,\cdot)$, $b(\cdot,\cdot)$, and $c(\cdot,\cdot)$ satisfy certain properties.  In the first case, we assume that $a(\cdot,\cdot)$ is coercive on $\mathbb M$.

\begin{theorem}(\cite{Xu.J2015a})
\label{thm:babuska1}
Assume that $a(\cdot,\cdot)$ is symmetric, $c(\cdot,\cdot)$ is symmetric positive semi-definite. Let $|\cdot|_e$ be a semi-norm on $\mathbb{N}$ such that $|p|_e\neq 0$, $\forall p\in \mathbb{N}\backslash\mathbb{K}$. Assume that the following inequalities
\begin{equation}
\label{eq:bounded_a}
a(u,\phi)\leq C_a\|u\|_{\mathbb{M}}\|\phi\|_{\mathbb{M}}, \forall {u},\phi\in {\mathbb M}
\end{equation}

\begin{equation}
\label{eq:bounded_b}
b(u,p)\leq C_b\|u\|_{\mathbb{M}}|p|_e, \forall {u}\in {\mathbb M}, p\in{\mathbb N}
\end{equation}

\begin{equation}
\label{eq:elliptic_a}
a(u,u)\geq \gamma_a\|u\|_{\mathbb{M}}^2, ~~\forall {u}\in {\mathbb M}
\end{equation}

\begin{equation}
\label{eq:infsup_b}
\sup_{u\in{\mathbb M}}\frac{b(u,q)}{\|u\|_{\mathbb{M}}}\geq \gamma_b|q|_e, \forall q\in\mathbb{N}
\end{equation}
hold with the constants $C_a$, $C_b$, $\gamma_a$ and $\gamma_b$ independent of parameters. In addition, assume that  $\forall q\in \mathbb{K}\backslash\{0\}$, $c(q,q)>0$.  Then, Problem \eqref{eq:saddle_abc} is uniformly well-posed with respect to parameters under the norms $\|\cdot\|_{\mathbb{M}}$ and $\|\cdot\|_{\mathbb{N}}$,
where $\|q\|_{\mathbb{N}}^2:=|\bar q|_{e}^2+|q|_c^2$ and $|q|_c^2=c(q,q)$.
\end{theorem}

\begin{proof}
Define
$$L(u,p; \phi, q) = a(u,\phi) + b(\phi, p) + b(u,q)  - c(p,q).$$

To prove well-posedness, we just need to verify the following:
\begin{itemize}
\item the boundedness of $L(\cdot;\cdot)$ under the norms $\|\cdot\|_{\mathbb{M}}$ and $\|\cdot\|_{\mathbb{N}}$.
\item the Babuska inf-sup condition: for any $(u, p) \in {\mathbb M} \times {\mathbb N}$,
\begin{equation} \label{eq:L-inf-sup}
\sup_{(\phi, q) \in {\mathbb M} \times {\mathbb N}} \frac{L(u,p;\phi,q)}{\left(  \|\phi\|_{\mathbb{M}}^2 + \| q \|_{\mathbb{N}}^2 \right)^{1/2}} \ge C(\| {u} \|_{\mathbb{M}}^2 + \| p \|_{\mathbb{N}}^2 )^{1/2}.
\end{equation}
\end{itemize}
Since it is straightforward to verify the boundedness of $L(\cdot ; \cdot)$, we focus on proving the inf-sup condition (\ref{eq:L-inf-sup}).

According to (\ref{eq:infsup_b}), for $\forall p \in {\mathbb N}$, consider its projection $\bar p \in {\mathbb K}^{\perp}$.  There exists $w$ such that
\begin{equation*}
b(w,\bar p) \ge \gamma_b | \bar p|_e^2 \ \text{and} \ \|w \|_{\mathbb{M}} = |\bar p |_{e}.
\end{equation*} 

Let $\phi ={u} + \theta w$, $\theta=\gamma_a\gamma_b/C_a^2$, and $q = -p$.  Then, we have
\begin{align*}
L(u,p; \phi, q) & = a({u},{u} + \theta w) + b({u} + \theta w, p) - b({u}, p) +  c(p,p) \\
		& \geq \gamma_a\|  u \|_{\mathbb{M}}^2 + \theta a( u,w) + \theta b(w,p) +  c(p,p) \\
		& \ge \gamma_a  \| u\|_{\mathbb{M}}^2 -\frac{\gamma_a}{2}\| u\|_{\mathbb{M}}^2- \frac{\theta^2C_a^2}{2\gamma_a } \|w\|_{\mathbb{M}}^2 + \gamma_b\theta |\bar p|_{e}^2 +  c(p,p) \\		
				& \ge \frac{\gamma_a}{2}\| u\|_{\mathbb{M}}^2 +\gamma_b\theta\left(1- \frac{\theta C_a^2}{2\gamma_a \gamma_b}\right) |\bar p|_{e}^2  +  c(p,p) \\		
		& = \frac{ \gamma_a}{2}  \| u\|_{\mathbb{M}}^2 +\frac{\gamma_a\gamma_b^2}{2C_a^2} |\bar p|_{e}^2 +  |p|_c^2. 
\end{align*}
Moreover, we have
\begin{align*}
\| \phi \|_{\mathbb{M}}^2 + |\bar q |_{e}^2  + |q|_c^2 &\leq 2\| u\|_{\mathbb{M}}^2+(2\gamma_a^2\gamma_b^2/C_a^4+1)|\bar p|_{e}^2+|p|_c^2. 
\end{align*}
Therefore, \eqref{eq:L-inf-sup} holds.

\end{proof}
\begin{remark}
It is worth noting that in case $\mathbb{K}=\{0\}$, we have $\bar q=q$ and $\|q\|_{\mathbb{N}}^2=|q|_{e}^2+|q|_c^2$.
\end{remark}

It is also possible to only assume that $a(\cdot,\cdot)$ is elliptic in ${\mathbb Z}$.  Then we need to assume that $c(\cdot,\cdot)$ is bounded under the norm $\|\cdot\|_{\mathbb{N}}$.

\begin{theorem}(\cite{Brezzi.F;Fortin.M1991a})
\label{thm:babuska2}
Assume that $a(\cdot,\cdot)$ and $c(\cdot,\cdot)$ are symmetric and positive semi-definite and that  (\ref{eq:bounded_a}), (\ref{eq:bounded_b}) and (\ref{eq:infsup_b}) hold.
Moreover, assume that
\begin{equation}
\label{eq:Zelliptic_a}
a( u, u)\geq \gamma_a\| u\|_{\mathbb{M}}^2, ~~\forall  u\in {\mathbb Z},
\end{equation}
\begin{equation}
\label{eq:Kelliptic_c}
c(q,q)\geq \gamma_c\|q\|_{\mathbb{N}}^2, ~~\forall q\in {\mathbb K},
\end{equation}
\begin{equation}
\label{eq:bounede_c}
c(p,q)\leq C_c\|p\|_{\mathbb{N}}\|q\|_{\mathbb{N}}, ~~\forall p,q\in {\mathbb N}.
\end{equation}
Assume that the constants $C_a$, $C_b$, $C_c$, $\gamma_a$, $\gamma_b$ and $\gamma_c$ are independent of the parameters. 
Then, Problem \eqref{eq:saddle_abc} is uniformly well-posed with respect to parameters under the norms $\|\cdot\|_{\mathbb{M}}$ and $\|\cdot\|_{\mathbb{N}}$.
\end{theorem}

Theorems \ref{thm:babuska1} and \ref{thm:babuska2} will be used to prove the well-posedness in different cases.  Note that they are sufficient conditions for the problems to be well-posed. For weaker conditions, we refer to \cite{Boffi.D;Brezzi.F;Fortin.M2013a}.

In this paper, we are especially interested in the robustness of preconditioners with respect to varying material and discretization parameters guided by the well-posedness of the linear system. Thus we want to emphasize the dependence on these parameters in inequalities. Therefore, we introduce the following notation: $\lesssim$, $\gtrsim$ and $\ucong$.  Given two quantities $x$ and $y$, $x\lesssim y$ means that there is a constant $C$ independent of these parameters such that $x\leq Cy$. $\gtrsim$ can be similarly defined. $x\ucong y$ if $x\lesssim y$ and $x\gtrsim y$.
\section{Relationship between preconditioning and well-posedness}
\label{sec:precondition}
Given that a variational problem is well-posed, an optimal precondtioner can be developed, in order to speed up Krylov subspace methods, such as Conjugate Gradient Method (CG) and Minimal Residual Method (MINRES). 
In order to illustrate this fact, we first consider the following variational problem:
\begin{description}
\item[]
Find $\bx\in \mathbb{X}$, such that
\begin{equation}
\label{eq:var_L}
\mathbf{L}(\bx,\by)=\langle \mathbf{f},\by\rangle,\quad \forall \by\in \mathbb{X},
\end{equation}
\end{description}
where $\mathbb{X}$ is a given Hilbert space and $\mathbf{f}\in \mathbb{X}'$.

The well-posedness of the variational problem (\ref{eq:var_L}) refers to the existence, uniqueness, and the stability $\|\bx\|_{\mathbb{X}}\lesssim \|\mathbf{f}\|_{\mathbb{X}'}$ of the solution. The necessary and sufficient conditions for (\ref{eq:var_L}) to be well-posed are shown in the following theorem.   We assume the symmetry $\mathbf{L}(\bx,\by)=\mathbf{L}(\by,\bx)$ in the rest of this section.

\begin{theorem}(\cite{Babuska.I1971a})
Problem (\ref{eq:var_L}) is well-posed if and only if the following conditions are satisfied:
\begin{itemize}
\item There exists a constant $C>0$ such that $\mathbf{L}(\bx,\by)\leq C\|\bx\|_{\mathbb{X}}\|\by\|_{\mathbb{X}}$.
\item There exists a constant $\beta>0$ such that
\begin{equation}
\label{eq:infsup_L}
\inf_{\bx\in \mathbb{X}}\sup_{\by\in \mathbb{X}}\frac{\mathbf{L}(\bx,\by)}{\|\bx\|_{\mathbb{X}}\|\by\|_{\mathbb{X}}}=\beta>0.
\end{equation}
\end{itemize}

\end{theorem}

%

Consider the operator form of (\ref{eq:var_L}):
$$\mathcal{L} \mathbf{x}=\mathbf{f}\in \mathbb{X}'.$$
Define operator $\mathcal{P}$ such that
\begin{equation}
\label{eq:riesz}
(\mathcal{P}\mathbf{f},\mathbf{y})_{\mathbb{X}}=\langle \mathbf{f},\by\rangle,\quad  \mathbf{f}\in \mathbb{X}', \by\in \mathbb{X}.
\end{equation}
Assuming the well-posedness, then the following inequalities hold
$$\|\mathcal{P}\mathcal{L}\|_{\mathbf{L}(\mathbb{X},\mathbb{X})}=\sup_{\bx,\by}\frac{(\mathcal{PL}\bx,\by)_{\mathbb X}}{\|\bx\|_{\mathbb{X}}\|\by\|_{\mathbb{X}}}= \sup_{\bx,\by}\frac{\langle\mathcal{L}\bx,\by\rangle}{\|\bx\|_{\mathbb{X}}\|\by\|_{\mathbb{X}}}\leq C,$$
$$\|(\mathcal{PL})^{-1}\|^{-1}_{\mathbf{L}(\mathbb{X},\mathbb{X})}=\inf_{\bx}\sup_\by\frac{(\mathcal{PL}\bx,\by)_{\mathbb X}}{\|\bx\|_{\mathbb{X}}\|\by\|_{\mathbb{X}}}=\inf_{\bx}\sup_\by\frac{\langle\mathcal{L}\bx,\by \rangle}{\|\bx\|_{\mathbb{X}}\|\by\|_{\mathbb{X}}}\geq \beta.$$
Therefore, the condition number of the precondtioned system is proved to be bounded
$$\kappa(\mathcal{PL}):=\|\mathcal{PL}\|_{\mathbf{L}(\mathbb{X},\mathbb{X})}\|(\mathcal{PL})^{-1}\|_{\mathbf{L}(\mathbb{X},\mathbb{X})}\leq C/\beta.$$
This type of preconditioners is frequently used in the literature and is characterized as ``mapping property'' in a recent review paper \cite{Mardal.K;Winther.R2011a}.

Let $\{\boldsymbol{\phi}_i\}$ be a set of given basis of $\mathbb{X}$ and $\{\boldsymbol{\phi}_i'\}$ be a set of given basis of $\mathbb{X'}$. Consider the matrix representation of $\mathcal{P}$ and $\mathcal{L}$ :
$$
\begin{aligned}
\mathcal{P}(\boldsymbol{\phi}_1',\cdots,\boldsymbol{\phi}_n')&=(\boldsymbol{\phi}_1,\cdots,\boldsymbol{\phi}_n)P,~~~ \mathcal{L}(\boldsymbol{\phi}_1,\cdots,\boldsymbol{\phi}_n)=(\boldsymbol{\phi}_1',\cdots,\boldsymbol{\phi}_n')L
\end{aligned}
$$
and the vector representation of $\bx$.
$$
 \bx=(\boldsymbol{\phi}_1,\cdots,\boldsymbol{\phi}_n)x.
$$
Assume $L$ is symmetric and $P$ is SPD.  Denote the mass matrix of $\mathbb{X}$ by $M$, i.e., $M_{ij}=(\boldsymbol{\phi}_i,\boldsymbol{\phi}_j)_{\mathbb{X}}$, $\forall i,j$. In fact, $P=M^{-1}$. Then
$$
\begin{aligned}
\|\mathcal{P}\mathcal{L}\|_{\mathbf{L}(\mathbb{X},\mathbb{X})}=\sup_{\bx,\by}\frac{(\mathcal{PL}\bx,\by)_{\mathbb X}}{\|\bx\|_{\mathbb{X}}\|\by\|_{\mathbb{X}}}= \sup_{x,y}\frac{x^T(PL)^TM y}{(x^TM x)^{1/2} (y^TM y)^{1/2}}
=\max_{\lambda\in \sigma(PL)}|\lambda|.\\
\end{aligned}
$$
Similarly,
$$\|(\mathcal{PL})^{-1}\|^{-1}_{\mathbf{L}(\mathbb{X},\mathbb{X})}=\min_{\lambda\in \sigma(PL)}|\lambda|.$$
Therefore, $\kappa(\mathcal{PL})=\kappa(PL)=\frac{\max\limits_{\lambda\in \sigma(PL)}|\lambda|}{\min\limits_{\lambda\in \sigma(PL)}|\lambda|}$.

A more general approach is via norm equivalence matrices \cite{Loghin.D;Wathen.A2004a}.  Given an SPD matrix $H$, $H$ inner product  and $H$ norm can be defined correspondingly:
$$(x ,x)_H:=(H x, x),\quad \|x\|^2_H:=(x,x)_H.$$

 Nonsingular matrices $A$ and $B$ are {\it H-norm equivalent}, denoted by $A\sim_{H}B$, if there are constants $\gamma$ and $\Gamma$ independent of the size of the matrices such that
$$\gamma\|B x\|_H\leq \|A x\|_H\leq \Gamma\|B x\|_H.$$  

If $A\sim_H B$ and $AB^{-1}$ is symmetric with respect to $(\cdot,\cdot)_H$, then MINRES preconditioned by $B^{-1}$ has the following convergence estimate \cite{Loghin.D;Wathen.A2004a}:
$$
\frac{\|r^k\|_H}{\|r^0\|_H}\leq 2\left(\frac{\Gamma-\gamma}{\Gamma+\gamma}\right)^{k/2}.
$$
Consider the preconditioner $P$ defined as the matrix representation of $\cP$ in (\ref{eq:riesz}). It is easy to see that $P^{-1}\sim_{M^{-1}} L$.  Note that $P=M^{-1}$.

This can help in the design of preconditioners for CG and MINRES.  Preconditioning GMRES differs in that it usually depends on the field of value analysis \cite{Loghin.D;Wathen.A2004a}. 

 In the rest of the paper, we will use Theorem \ref{thm:babuska1} and Theorem \ref{thm:babuska2} to prove the well-posedness of the different formulations of the Biot model under different choices of $\mathbb{X}$. Then, based on the well-posedness, we show the corresponding optimal block preconditioners.

\section{A two-field formulation}
\label{sec:2field}
The preconditioning for the two-field system (\ref{eq:2by2}) has been studied extensively in the literature  \cite{Phoon.K;Toh.K;Chan.S;Lee.F2002a,Toh.K;Phoon.K;Chan.S2004a,Haga.J;Osnes.H;Langtangen.H2012a,Haga.J;Osnes.H;Langtangen.H2011a}, where the Schur complement approach is usually used to develop preconditioners.  In this paper, similar to \cite{adler2017robust}, we briefly formulate a preconditioner based on the well-posedness of the linear systems for the two-field Biot model.

We first study the well-posedness of (\ref{eq:2by2}), beginning by changing the variable $\tilde p=-\alpha p$ in order to symmetrize (\ref{eq:2by2}). With an abuse of notation, we still use the notation $p$ for pressure after the change of variable.  Next, we introduce the function space for the displacement and the pressure.  Due to the boundary conditions (\ref{eq:bdc_u}), we consider
$$\mathbb{U}\subset H^1_D(\Omega):=\{\mathbf{u}\in (H^1(\Omega))^n| \mathbf{u}=0, \mbox{ on }\Gamma_{D,u}\}$$
 for the displacement and $$\mathbb{Q}_c\subset H^1_P(\Omega):=\{p\in H^1(\Omega)| p=0, \mbox{ on }\Gamma_{N,v}\}$$ for the pressure. Here, we use the subscript ``c'' to suggest the continuity of the functions in $\mathbb{Q}_c$. We assume $|\Gamma_{D,u}|>0$ in the rest of this paper so that the elasticity operator is nonsingular on $\mathbb{U}$.  We also assume that $|\Gamma_{N,u}|>0$ such that the divergence operator is surjective on the pressure space. 

Then, we define the following bilinear forms:
\begin{equation*}
\begin{aligned}
\mbox{for } \mathbf{u}, \boldsymbol{\phi}\in \mathbb{U},\quad &a^I(\mathbf{u},\boldsymbol{\phi})= (2\mu\epsilon(\mathbf{u}),\epsilon(\boldsymbol{\phi}))+(\lambda\nabla\cdot \mathbf{u},\nabla\cdot \boldsymbol{\phi}),\\
\mbox{for } \mathbf{u} \in {\mathbb U}, ~~p\in \mathbb{Q}_c,\quad &b^{I}(\mathbf{u},p)=(\nabla\cdot \mathbf{u},p),\\
\mbox{for } p,q\in {\mathbb Q}_c,\quad &d^I(p,q) = (\kappa^{-1}\nabla p,\nabla q) + (\xi p,q),\\
\end{aligned}
\end{equation*}
where $\kappa = \alpha^2/(\Delta t k)$ and $\xi = S/\alpha^2$.

Now, we introduce the notation for the kernel spaces:
$${\mathbb Z^I}=\{\mathbf{u}\in{\mathbb U}| b^I(\mathbf{u},q)=0, \forall q\in {\mathbb Q}_c\},~~~{\mathbb K^I} = \{p\in \mathbb{Q}_c| b^I(\boldsymbol{\phi},p)=0,\forall \boldsymbol{\phi}\in\mathbb{U}\}.$$

The variational formulation of (\ref{eq:2by2}) is as follows:

Find $(\mathbf{u},p)\in {\mathbb U}\times{\mathbb Q}_c$ such that $\forall (\boldsymbol{\phi},q)\in {\mathbb U}\times {\mathbb Q}_c$, the following equations hold

\begin{equation}
\label{eq:saddle_2by2}
\left\{
\begin{aligned}
a^I(\mathbf{u},\boldsymbol{\phi})&+ b^I(\boldsymbol{\phi},p)&=& (f,\boldsymbol{\phi}),\\
b^I(\mathbf{u},q)&-d^I(p,q)&=&(g,q).\\
\end{aligned}
\right.
\end{equation}

We define the norms as follows:

\begin{equation}
\label{eq:norm_up}
\begin{aligned}
\|\mathbf{u}\|_{\mathbb{U}}^2 =& a^I(\mathbf{u},\mathbf{u}),~~~\|q\|_{\mathbb{Q}_c}^2=\beta^{-1}\|q\|^2_0+d^I(q,q),
\end{aligned}
\end{equation}
where $\beta=\max\left\{{\mu},{\lambda}\right\}$.

This variational formulation (\ref{eq:saddle_2by2}) is proved to be well-posed under the norms $\|\cdot\|_{\mathbb{U}}$ and $\|\cdot\|_{\mathbb{Q}_c}$
provided that the following inf-sup condition holds

\begin{equation}
\label{eq:infsup_up}
\forall p\in (\mathbb{K}^I)^{\perp},\quad \sup_{\mathbf{u}\in\mathbb{U}}\frac{b^I(\mathbf{u},p)}{\|\mathbf{u}\|_1}\gtrsim \|p\|_0.
\end{equation}
It is well known that (\ref{eq:infsup_up}) holds for ${\mathbb U}=H_D^1(\Omega)$, ${\mathbb Q}=L^2(\Omega)$ and $\mathbb{U}=(H^1_0(\Omega))^n$, $\mathbb{Q}=L^2_0(\Omega)$  on a bounded domain $\Omega$ with Lipschitz boundary \cite{Bramble.J;Lazarov.R;Pasciak.J2001a,Bramble.J2003a}.  Moreover, (\ref{eq:infsup_up}) holds for stable Stokes FEM pairs \cite{Boffi.D;Brezzi.F;Fortin.M2013a}.

\begin{theorem}
Assume that the inf-sup conditions (\ref{eq:infsup_up}) holds and $\beta=\max\left\{{\mu},{\lambda}\right\}$. The system (\ref{eq:saddle_2by2}) is uniformly well-posed with respect to parameters under the norms $\|\cdot\|_{\mathbb{U}}$ and $\|\cdot\|_{\mathbb{Q}_c}$defined in  (\ref{eq:norm_up}).
\end{theorem}

\begin{proof}
To prove the well-posedness, we just need to verify the assumptions of Theorem \ref{thm:babuska1}. 

As we assume that $|\Gamma_{N,u}|>0$, we know that $\mathbb{K}^{I}=\{0\}$ and then (\ref{eq:Kelliptic_c}) is trivial.
By definition, (\ref{eq:bounded_a}), (\ref{eq:bounded_b}), and (\ref{eq:elliptic_a}) are straightforward to verify.

Based on (\ref{eq:infsup_up}), the following inf-sup condition is implied
\begin{equation}
\label{}
\forall p\in (\mathbb{K}^I)^{\perp},\quad \sup_{\mathbf{u}\in \mathbb{U}}\frac{b^I(\mathbf{u},p)}{\|\mathbf{u}\|_{\mathbb{U}}}\gtrsim \|p\|_{\mathbb{Q}}.
\end{equation}
Then (\ref{eq:infsup_b}) is verifed.

Therefore,  the proof is finished by applying Theorem \ref{thm:babuska1}.

\end{proof}

With the well-posedness  of (\ref{eq:saddle_2by2}) proved,  an optimal preconditioner can be formulated. 
We first introduce some matrix notation. Given finite element basis functions $\{\mathbf{u}_i\}$ and $\{p_i\}$ for $\mathbb{U}$ and $\mathbb{Q}$, respectively, define the following stiffness matrices:  $(A_u)_{ij}:=a^I(\mathbf{u}_i,\mathbf{u}_j)$, $(B_u)_{ij}=b^I(\mathbf{u}_i,p_j)$, $(A_p)_{ij}=d^I(p_i,p_j)$ and $(M_p)_{ij}=(p_i,p_j)$.

The matrix form of the system and preconditioner are

$$
S^{II}=\left(
\begin{array}{cc}
A_u & B_u^T\\
B_u&-A_p \\
\end{array}
\right)
~\mbox{ and }~
P^{II}=
\left(
\begin{array}{cc}
A_u& \\
& \beta^{-1}M_p+ A_p\\
\end{array}
\right)^{-1},
$$
respectively.

\begin{remark}
In case $|\Gamma_{N,u}|=0$, the kernel space $\mathbb{K}^I$ contains constant functions.  We can similarly prove the well-posedness, but the norm $\|q\|_{\mathbb{N}}$ has a term $|\bar q|_{e}$, which results in a dense matrix in the preconditioner.  We refer to \cite{Lee.J;Mardal.K;Winther.R2015a} for constructing preconditioners related to $|\bar q|_e$. 


%
\end{remark}

In the literature, the preconditioners for two-field formulation are mostly based on Schur complement approaches.  The exact Schur complement preconditioner of $S^{II}$, i.e.,
$$
\left(
\begin{array}{cc}
A_u & \\
&A_p+B_uA_u^{-1}B_u^T \\
\end{array}
\right),
$$
is known to be an optimal preconditioner \cite{Elman.H;Silvester.D;Wathen.A2005a}, although $B_uA_u^{-1}B_u^T$ is dense and cannot be obtained. Practical approximations of $B_uA_u^{-1}B_u^T$, such as
$$B_u\mbox{diag}(A_u)^{-1}B_u^T~\mbox{ and }~ \mbox{diag}(B_u\mbox{diag}(A_u)^{-1}B_u^T),$$
have also been investigated \cite{Phoon.K;Toh.K;Chan.S;Lee.F2002a,Toh.K;Phoon.K;Chan.S2004a,Haga.J;Osnes.H;Langtangen.H2012a,Haga.J;Osnes.H;Langtangen.H2011a}.

The two-field formulation is usually considered computationally efficient, as it involves the fewest variables and, therefore, has smaller linear systems to solve than the three-field formulation (\ref{eq:biot_3field}).  However, the two-field formulation (with continuous pressure elements) exhibits oscillations in the pressure field, and more expanded systems such as the three-field formulation, are shown to be more stable \cite{Ferronato.M;Castelletto.N;Gambolati.G2010a,Haga.J;Osnes.H;Langtangen.H2012b}.  Motivated by this fact, we study a three-field formulation \cite{Haga.J;Osnes.H;Langtangen.H2012b} in the next section.

\section{A three-field formulation}\label{sec:3field}
In this section, we will show the well-posedness of the three field formulation (\ref{eq:biot_3field}), briefly formulate the diagonal block robust preconditioners of \cite{hong2018parameter,adler2019robust} as special cases, and propose some new preconditioners for the three field formulation guided by the well-posedness.
\subsection{The three-field formulation}
We can write (\ref{eq:biot_3field}) as a symmetric problem by rescaling. Introduce $$\tilde{\mathbf{v}}=\frac{\Delta t}{\alpha}\mathbf{v},\quad \tilde p=-\alpha p.$$
The three-field system (\ref{eq:biot_3field}) can be rewritten as
\begin{equation}
\label{eq:3by3}
\left\{
\begin{aligned}
-\nabla\cdot(2 \mu \epsilon(\mathbf{u}) + \lambda (\nabla \cdot  \mathbf{u})  \bI)-\nabla \tilde p&=f,\\
\kappa\tilde{\mathbf{v}}-\nabla \tilde p&=r,\\
\nabla\cdot \mathbf{u}+\nabla\cdot \tilde{\mathbf{v}}-\xi \tilde p&=g.\\
\end{aligned}
\right.
\end{equation}

%
With an abuse of notation, we still use $\mathbf{v}$ and $p$ to denote the scaled velocity $\tilde{\mathbf{v}}$, the scaled pressure $\tilde{p}$, respectively.
Then, we introduce the function spaces:
$$\mathbb{V}\subset H_{D}(\mbox{div},\Omega):=\{\mathbf{v}\in H(\mbox{div},\Omega)|\mathbf{v}\cdot \mathbf {n}= 0, \mbox{ on } \Gamma_{D,v}\},$$
$$\mathbb{W}={\mathbb U}\times  {\mathbb V},~~~\mathbb{Q}\subset L^2(\Omega),$$
and bilinear forms
$$\mbox{for } (\mathbf{u},\mathbf{v}),(\boldsymbol{\phi},\boldsymbol{\psi})\in \mathbb{W},\quad a^{II}(\mathbf{u},\mathbf{v}; \boldsymbol{\phi},\boldsymbol{\psi})= a^I(\mathbf{u},\boldsymbol{\phi})+ (\kappa\mathbf{v},\boldsymbol{\psi}),$$
$$\mbox{for } (\mathbf{u},\mathbf{v}) \in {\mathbb W}, ~~p\in \mathbb{Q},\quad b^{II}(\mathbf{u},\mathbf{v};p)=b^I(\mathbf{u},p)+ (\nabla\cdot\mathbf{v},p),$$
$$\mbox{for } p,q\in {\mathbb Q}, ~~c^I(p,q)=(\xi p,q),\quad \xi> 0.$$

We define the corresponding kernel spaces related to $b^{II}(\cdot;\cdot)$
$${\mathbb Z}^{II}=\{(\mathbf{u},\mathbf{v})\in{\mathbb W}| b^{II}(\mathbf{u},\mathbf{v};q)=0, \forall q\in {\mathbb Q}\},$$
$${\mathbb K^{II}} = \{p\in \mathbb{Q}| b^{II}(\boldsymbol{\phi},\boldsymbol{\psi};p)=0,\forall (\boldsymbol{\phi},\boldsymbol{\psi})\in\mathbb{W}\}.$$
Note that due to the assumption $|\Gamma_{N,u}|>0$, we have $\mathbb{K}^{II}=\{0\}$.

Then, the weak formulation is as follows:

Find $(\mathbf{u},\mathbf{v})\in {\mathbb W}$ and $p\in{\mathbb Q}$ such that $\forall (\boldsymbol{\phi},\boldsymbol{\psi})\in {\mathbb W}$ and $q\in {\mathbb Q}$, the following equations hold

\begin{equation}
\label{eq:saddle_3}
\left\{
\begin{aligned}
a^{II}(\mathbf{u},\mathbf{v};\boldsymbol{\phi},\boldsymbol{\psi})&+ b^{II}(\boldsymbol{\phi},\boldsymbol{\psi};p)&=& (f,\boldsymbol{\phi})+(r, \boldsymbol{\psi}),\\
b^{II}(\mathbf{u},\mathbf{v};q)&-c^I(p,q)&=&(g,q).\\
\end{aligned}
\right.
\end{equation}
The additional term $c^I(p,q)$ corresponds to different versions of the Biot models \cite{Axelsson.O;Blaheta.R;Byczanski.P2012a}. 

The well-posedness of this saddle point problem can be proved with different choices of norms for ${\mathbb W}$ and $\mathbb{Q}$. We discuss some of these options in the rest of this section.

\subsection{Augmented Lagrangian preconditioners}
The stability of the three-field system (\ref{eq:3by3}) is closely related to the stability of the pair $\mathbf{u}$-$p$ and $\bv$-$p$. In particular, it is considered stable if $\mathbf{u}$-$p$ satisfies  (\ref{eq:infsup_up})
and $\mathbf{v}$-$p$ satisfies
\begin{equation}
\label{eq:infsup_vp}
\forall p\in (\mathbb{K}_v)^{\perp},\quad \quad \sup_{\mathbf{v}\in\mathbb{V}}\frac{(\nabla\cdot\mathbf{v},p)}{\|\mathbf{v}\|_{H({\rm div})}}\gtrsim \|p\|_0,
\end{equation}
where 
$$\mathbb{K}_v:= \{p\in \mathbb{Q} |(\nabla \cdot\mathbf{v},p)=0,   \forall \mathbf{v}\in \mathbb{V} \}.$$
  (\ref{eq:infsup_vp}) holds for $\mathbb{V}=H_D(\mbox{div},\Omega)$ and $\mathbb{Q}=L^2(\Omega)$ and, in discrete cases, there are many stable pairs, such as Raviart-Thomas elements \cite{Raviart.P;Thomas.J1977a} for $\mathbb{V}$ and piecewise polynomials for $\mathbb{Q}$.

The augmented Lagrangian (AL) method \cite{Benzi.M;Olshanskii.M2006a,Xu.J;Yang.K2015a} incorporates the constraint into the norm.  The constraint here is $$\nabla\cdot(\mathbf{u}+\mathbf{v})=0.$$
 Therefore, it is natural to consider the following norms for the AL method.

Let $P_Q$ be the $L^2$ projection from $L^2(\Omega)$ to $\mathbb Q$. We define the norms for spaces ${\mathbb W}$ and $\mathbb{Q}$ as follows:
\begin{equation}
\label{eq:norm_al}
\begin{aligned}
\|\mathbf{v}\|_{\mathbb{V}}^2=&(\kappa\mathbf{v},\mathbf{v}),\\
\|(\mathbf{u},\mathbf{v})\|_{\mathbb{W}}^2=& \|\mathbf{u}\|_{\mathbb{U}}^2+\|\mathbf{v}\|_{\mathbb{V}}^2+\beta\|P_Q\nabla\cdot(\mathbf{u}+\mathbf{v})\|_0^2,\\
\|q\|_{\mathbb{Q}}^2=&(\beta^{-1} q,q),\\
\end{aligned}
\end{equation}
where $\xi$ is the coefficient in bilinear form $c^I(\cdot,\cdot)$, and $\beta$ is an undetermined parameter. 

To prove the well-posedness of (\ref{eq:saddle_3}), we just need to verify the assumptions of Theorem \ref{thm:babuska1}.

Given (\ref{eq:infsup_up}), we have $\forall q\in ({\mathbb K}^{II})^{\perp},$
\begin{equation}
\label{eq:InfSup_up}
\sup_{\mathbf{u},\mathbf{v}}\frac{b^{II}(\mathbf{u},\mathbf{v};q)}{||(\mathbf{u},\mathbf{v})||_{\mathbb{W}}} \geq \sup_{\mathbf{u}}\frac{(\nabla \cdot \mathbf{u},q)}{(||\mathbf{u}||^2_{\mathbb{U}}+\beta\|P_Q\nabla\cdot \mathbf{u}\|_0^2)^{1/2}}\gtrsim \max\{\mu,\lambda,\beta\}^{-1/2}\|q\|_0.
\end{equation}
For the case in which $\beta\geq \max\{\mu,\lambda\}$, the right-hand side of (\ref{eq:InfSup_up}) is equal to $\|q\|_{\mathbb{Q}}$.

Given (\ref{eq:infsup_vp}), we can prove another inequality: $\forall q\in ({\mathbb K}^{II})^{\perp}$
\begin{equation}
\label{eq:InfSup_vp}
 \sup_{\mathbf{u},\mathbf{v}}\frac{b^{II}(\mathbf{u},\mathbf{v};q)}{||(\mathbf{u},\mathbf{v})||_{\mathbb{W}}} \geq \sup_{\mathbf{v}}\frac{(\nabla \cdot\mathbf{v},q)}{(||\mathbf{v}||^2_{\mathbb{V}}+\beta\|P_Q\nabla\cdot\mathbf{v}\|_0^2)^{1/2}}\gtrsim \max\{\kappa,\beta\}^{-1/2}\|q\|_0.
\end{equation}
Similarly, if we further assume that $\beta\geq \kappa $, the right-hand side of (\ref{eq:InfSup_vp}) is equal to $\|q\|_{\mathbb{Q}}$.  Note that this approach is used in \cite{Lipnikov.K2002a}, where the displacement $\bu$ is set to be zero and the inf-sup condition of the $\bv$-$p$ pair is assumed. 

The boundedness of $b^{II}(\cdot,\cdot)$ is easy to verify due to the additional term $\beta\|P_Q\nabla\cdot(\mathbf{u}+\mathbf{v})\|_0$ in the norm $\|\cdot\|_{\mathbb{W}}$:
\begin{equation}
\label{eq:bounded_bII}
b^{II}(\mathbf{u},\mathbf{v};q) \leq \|P_Q\nabla\cdot(\mathbf{u}+\mathbf{v})\|_0\|q\|_0\leq \|(\mathbf{u},\mathbf{v})\|_{\mathbb{W}}\|q\|_{\mathbb{Q}}.
\end{equation}
The coercivity of $a^{II}(\cdot,\cdot)$ is straightforward to prove, as 
\begin{equation}
\label{eq:coercive_bounded_a}
\forall (\mathbf{u},\mathbf{v})\in {\mathbb Z}^{II},\quad a^{II}(\mathbf{u},\mathbf{v};\mathbf{u},\mathbf{v})\equiv \|(\mathbf{u},\mathbf{v})\|_{\mathbb{W}}^2.
\end{equation}

Because $a^{II}(\cdot,\cdot)$ is uniformly coercive only on $\mathbb{Z}^{II}$, we resort to Theorem \ref{thm:babuska2} to prove the well-posedness.

\begin{theorem}
Assume  $\beta=\min\{\max\{\mu,\lambda\},\kappa\}$, $\xi\beta$ is uniformly bounded and the inf-sup conditions (\ref{eq:infsup_up}) and (\ref{eq:infsup_vp}) hold. 
Then the system (\ref{eq:3by3}) is uniformly well-posed with respect to parameters under the norms $\|\cdot\|_{\mathbb{W}}$ and $\|\cdot\|_{\mathbb{Q}}$ defined in (\ref{eq:norm_al}).
\end{theorem}
\begin{proof}
As $\mathbb{K}^{II}=\{0\}$, (\ref{eq:Kelliptic_c}) is trivial to prove. Consider $\bar q$ for the inf-sup condition of $b^{II}(\cdot,\cdot)$. Due to $\beta=\min\{\max\{\mu,\lambda\},\kappa\}$, the right-hand side of (\ref{eq:InfSup_up}) or (\ref{eq:InfSup_vp}) is equal to $\|\bar q\|_{\mathbb{Q}}$. Therefore, the inf-sup condition of $b^{II}(\cdot,\cdot)$ is proved. 


As $0<\xi\lesssim \beta^{-1}$,  we can prove that $c^I(p,q)\lesssim \|p\|_{\mathbb{Q}}\|q\|_{\mathbb{Q}}$.  Therefore, the assumptions of Theorem \ref{thm:babuska2} hold.  Then the proof is finished by applying Theorem \ref{thm:babuska2}.
\end{proof}

It is obvious that we only need to assume either (\ref{eq:infsup_up}) or (\ref{eq:infsup_vp}) to prove the well-posedness of (\ref{eq:saddle_3}).
\begin{coro}
 Assume $\beta=\max\{\mu,\lambda\}$, $\xi\beta$ is uniformly bounded, and that the inf-sup condition (\ref{eq:infsup_up}) holds. The system (\ref{eq:3by3}) is uniformly well-posed with respect to parameters under the norms defined in (\ref{eq:norm_al}). 
\end{coro}
\begin{proof}
The proof follows from (\ref{eq:InfSup_up}), (\ref{eq:coercive_bounded_a}),  (\ref{eq:bounded_b}), and Theorem \ref{thm:babuska2}.
\end{proof}

\begin{coro}
Assume that $\beta=\kappa$, $\xi\beta$ is uniformly bounded, and the inf-sup condition (\ref{eq:infsup_vp}) holds. The system (\ref{eq:3by3}) is uniformly well-posed with respect to parameters under the norms defined in (\ref{eq:norm_al}). 
\end{coro}
\begin{proof}
The proof follows from (\ref{eq:InfSup_vp}), (\ref{eq:coercive_bounded_a}),  (\ref{eq:bounded_b}), and Theorem \ref{thm:babuska2}.
\end{proof}

\begin{remark}
The assumption that both (\ref{eq:infsup_up}) and (\ref{eq:infsup_vp}) hold results in a smaller parameter $\beta$ than the cases where only one of (\ref{eq:infsup_up}) and (\ref{eq:infsup_vp}) holds.
\end{remark}

Based on the well-posed formulation, we derive the corresponding optimal block diagonal preconditioner.

\subsubsection{Matrix form}
We introduce some additional matrix notation.  Also, we introduce the FEM basis $\{\mathbf{v}_i\}$ for $\mathbb{V}$.  Define the stiffness matrices $(M_v)_{ij}=(\mathbf{v}_i,\mathbf{v}_j)$, $(A_v)_{ij}=(\kappa\mathbf{v}_i,\mathbf{v}_j)$, $(C_p)_{ij}=c^I(p_i,p_j)$, and $(B_v)_{ij}=(\nabla\cdot \mathbf{v}_i,p_j)$.

Then the system matrix of the three-field formulation is 

$$
S^{III}=
\left(
\begin{array}{ccc}
A_u & &B_u^T\\
& A_v& B_v^T\\
B_u&B_v&-C_p \\
\end{array}
\right).
$$

The block preconditioner is 
$$P_1^{III}
=
\left(
\begin{array}{ccc}
A_u+\beta B_u^TM_p^{-1}B_u &\beta B_u^TM_p^{-1}B_v&\\
\beta B_v^TM_p^{-1}B_u&A_v+\beta B_v^TM_p^{-1}B_v&\\
&& {\beta^{-1}}M_p+C_p\\
\end{array}
\right)^{-1}.
$$ 

%
In order to be uniformly optimal with respect to the parameters,  $\beta$ is chosen as follows:
\begin{itemize}
\item $\beta=\max\{\mu,\lambda\}$, if (\ref{eq:infsup_up}) holds;
\item $\beta = \kappa$, if (\ref{eq:infsup_vp}) holds;
\item $\beta = \min\{\max\{\mu,\lambda\},\kappa\}$, if both (\ref{eq:infsup_up}) and (\ref{eq:infsup_vp}) hold.
\end{itemize}

\subsection{Block diagonal preconditioners}
We can formulate block diagonal preconditioners based on (\ref{eq:infsup_up}).  Define another pair of norms for spaces ${\mathbb W}$ and $\mathbb{Q}$:
\begin{equation}
\label{eq:norm1}
\begin{aligned}
\|(\mathbf{u},\mathbf{v})\|_{\mathbb{W}}^2=& \|\mathbf{u}\|_{\mathbb{U}}^2+\|\mathbf{v}\|_{\mathbb{V}}^2+\beta\|P_Q\nabla\cdot\mathbf{v}\|_0^2,~~~\|q\|_{\mathbb{Q}}^2=\beta^{-1}\|q\|^2_0,
\end{aligned}
\end{equation}
where $\beta=\min\{\max\left\{{\mu},{\lambda}\right\},\kappa\}$.

\begin{theorem}
\label{thm:divv}
Assume  $\beta=\min\{\max\{\mu,\lambda\},\kappa\}$, $\xi\beta$ is uniformly bounded and the inf-sup conditions (\ref{eq:infsup_up}) and (\ref{eq:infsup_vp}) hold. 
Then the system (\ref{eq:3by3}) is uniformly well-posed with respect to parameters under the norms $\|\cdot\|_{\mathbb{W}}$ and $\|\cdot\|_{\mathbb{Q}}$ defined in (\ref{eq:norm1}).
\end{theorem}
\begin{proof}
We need to verify the assumptions of Theorem \ref{thm:babuska2} in order to finish the proof. The inf-sup condition of $b^{II}(\cdot,\cdot)$ follows from (\ref{eq:InfSup_up}) or (\ref{eq:InfSup_vp}) and the assumption that $\beta=\min\{\max\{\mu,\lambda\},\kappa\}$.  The boundedness of $b^{II}(\cdot,\cdot)$ can be shown to be uniform:
$$
\begin{aligned}
&b^{II}(\mathbf{u},\mathbf{v}; p)\leq (\|\nabla\cdot \mathbf{u}\|_0+\|\nabla\cdot\mathbf{v}\|_0)\|p\|_0\\
\lesssim &\|\mathbf{u}\|_{\mathbb{U}}{\beta^{-1/2}}\|p\|_0+ \|\mathbf{v}\|_{\mathbb{V}}{\beta^{-1/2}}\|p\|_0 \lesssim \|(\mathbf{u},\mathbf{v})\|_{\mathbb{W}}\|p\|_{\mathbb{Q}}.
\end{aligned}
$$

In the kernel $\mathbb{Z}^{II}$ we have $P_{Q}\nabla\cdot \mathbf{u} =-P_Q\nabla\cdot\mathbf{v}$; therefore, the coercivity can be shown as $\forall (\mathbf{u},\mathbf{v})\in\mathbb{Z}^{II}$
$$
\begin{aligned}
&a^{II}(\mathbf{u},\mathbf{v}; \mathbf{u},\mathbf{v})\gtrsim a^{II}(\mathbf{u},\mathbf{v};\mathbf{u},\mathbf{v})+ \beta\|\nabla\cdot \mathbf{u}\|_0^2\\
\geq& a^{II}(\mathbf{u},\mathbf{v};\mathbf{u},\mathbf{v})+\beta\|P_Q\nabla\cdot\mathbf{v}\|_0^2= \|(\mathbf{u},\mathbf{v})\|_{\mathbb{W}}^2.
\end{aligned}
$$

The boundedness of $a^{II}(\cdot,\cdot)$ and the assumptions on $c^I(\cdot,\cdot)$ are straightforward to verify.
\end{proof}

\begin{coro}
 Assume that the inf-sup condition (\ref{eq:infsup_up}) holds,  $\beta=\max\left\{{\mu},{\lambda}\right\}$ and $\xi\beta$ is uniformly bounded. The system (\ref{eq:3by3}) is uniformly well-posed with respect to parameters under the norms $\|\cdot\|_{\mathbb{W}}$ and $\|\cdot\|_{\mathbb{Q}}$ defined in (\ref{eq:norm1}).
\end{coro}
\begin{proof}
The proof follows from  Theorem \ref{thm:babuska2} and the inf-sup condition of $b^{II}(\cdot,\cdot)$ follows from (\ref{eq:InfSup_up}). And the boundedness of $b^{II}(\cdot,\cdot)$ and coercivity of $a^{II}(\cdot, \cdot)$ can be obtained similarly to the proof of Theorem \ref{thm:divv}.
\end{proof}

\begin{coro}
Assume that $\beta=\kappa$, $\xi\beta$ is uniformly bounded, and the inf-sup condition (\ref{eq:infsup_vp}) holds. The system (\ref{eq:3by3}) is uniformly well-posed with respect to parameters under the norms defined in (\ref{eq:norm1}). 
\end{coro}
\begin{proof}
The proof follows from Theorem \ref{thm:babuska2} and the inf-sup condition of $b^{II}(\cdot,\cdot)$ follows from (\ref{eq:InfSup_vp}). And the boundedness of $b^{II}(\cdot,\cdot)$ and coercivity of $a^{II}(\cdot, \cdot)$ can be obtained similarly to the proof of Theorem \ref{thm:divv}.

\end{proof}

\begin{remark}
The assumption that both (\ref{eq:infsup_up}) and (\ref{eq:infsup_vp}) hold results in a smaller parameter $\beta$ than the cases where only one of (\ref{eq:infsup_up}) and (\ref{eq:infsup_vp}) holds. 
\end{remark}

%

\subsubsection{Matrix form}
The matrix form of the block diagonal preconditioner is as follows:

$$P_2^{III}
=
\left(
\begin{array}{ccc}
A_u &&\\
&A_v+\beta B_v^TM_p^{-1}B_v&\\
&& {\beta^{-1}}M_p+C_p\\
\end{array}
\right)^{-1},
$$ 
where $\beta=\min\{\max\{\mu,\lambda\},\kappa\}$.

%


In the previous approach, we added $\beta(\nabla \cdot\mathbf{v},\nabla\cdot\mathbf{v})$ to the norm. This term causes some difficulty for the block solvers when $\beta$ is large.
There are lots of studies on this topic, and we may resort to Hiptmair-Xu preconditioners \cite{Hiptmair.R;Xu.J2007a}.  Here, we can avoid this term by adding a Laplace-like term on the pressure diagonal block. This is also used in the mixed formulation for Poisson equations.

We define the norm for spaces ${\mathbb W}$ and $\mathbb{Q}$:

\begin{equation}
\label{eq:norm2}
\begin{aligned}
\|(\mathbf{u},\mathbf{v})\|_{\mathbb{W}}^2&= \|\mathbf{u}\|_{\mathbb{U}}^2+\|\mathbf{v}\|_{\mathbb{V}}^2,~~~\|q\|_{\mathbb{Q}}^2=\beta^{-1}\| q\|^2_0+\|\mbox{div}^*_Vq\|_{\mathbb{V}'}^2,
\end{aligned}
\end{equation}
where $\beta=\max\{\mu,\lambda\}$ and  $\mbox{div}_V^{*}:{\mathbb Q}\mapsto{\mathbb V}'$ is the adjoint operator of $\mbox{div}_V:{\mathbb V}\mapsto{\mathbb Q}'$; i.e.,
$$\langle\mbox{div}_V^*q,\mathbf{v}\rangle:=\langle q,\mbox{div}_V\mathbf{v}\rangle,~~\forall \mathbf{v}\in \mathbb{V}, q\in \mathbb{Q}.$$

\begin{theorem}
Assume that the inf-sup conditions (\ref{eq:infsup_up}) and (\ref{eq:infsup_vp}) hold and $\beta=\max\left\{{\mu},{\lambda}\right\}$. The system (\ref{eq:3by3}) is uniformly well-posed with respect to parameters under the norms in (\ref{eq:norm2}).
\end{theorem}
\begin{proof}
As $a^{II}(\cdot,\cdot)$ is coercive on ${\mathbb W}$ due to the fact that
$$a^{II}(\mathbf{u},\mathbf{v};\mathbf{u},\mathbf{v})\equiv \|(\mathbf{u},\mathbf{v})\|_{\mathbb{W}}^2, ~\forall (\mathbf{u},\mathbf{v})\in {\mathbb W},$$
we can use Theorem \ref{thm:babuska1} to finish the proof.

First, we consider the inf-sup condition of $b^{II}(\cdot,\cdot)$.

Given that $q\in \mathbb{Q}$, we have the following inequalities:
$$\sup_{(\mathbf{u},\mathbf{v})}\frac{b^{II}(\mathbf{u},\mathbf{v};q)}{||(\mathbf{u},\mathbf{v})||_{\mathbb{W}}} \geq \sup_{\mathbf{u}}\frac{(\nabla \cdot \mathbf{u},q)}{||\mathbf{u}||_{\mathbb{U}}}\gtrsim  \sup_{\mathbf{u}}\frac{(\nabla \cdot \mathbf{u},q)}{\beta^{1/2}||\mathbf{u}||_1}\gtrsim \beta^{-1/2}\|q\|_0,$$

$$\sup_{(\mathbf{u},\mathbf{v})}\frac{b^{II}(\mathbf{u},\mathbf{v};q)}{||(\mathbf{u},\mathbf{v})||_{\mathbb{W}}} \geq \sup_{\mathbf{v}}\frac{(\nabla \cdot\mathbf{v},q)}{||\mathbf{v}||_{\mathbb{V}}} = \|\mbox{div}^*_V q\|_{\mathbb{V}'}.$$

Therefore, we have
$$\sup_{(\mathbf{u},\mathbf{v})}\frac{b^{II}(\mathbf{u},\mathbf{v};q)}{||(\mathbf{u},\mathbf{v})||_{\mathbb{W}}} \gtrsim \|q\|_{\mathbb{Q}}.$$

The boundedness of $b^{II}(\cdot,\cdot)$ can be shown to be uniform:
$$
\begin{aligned}
&b^{II}(\mathbf{u},\mathbf{v}; p)=(\nabla\cdot \mathbf{u},p)+(\nabla\cdot\mathbf{v},p)\\
\leq& \|\mathbf{u}\|_{\mathbb{U}}\frac{1}{\beta^{1/2}}\|p\|_0+\|\mbox{div}^* p\|_{\mathbb{V}'}\|\mathbf{v}\|_{\mathbb{V}}\lesssim \|(\mathbf{u},\mathbf{v})\|_{\mathbb{W}}\|p\|_{\mathbb{Q}}.
\end{aligned}
$$

The boundedness of $a^{II}(\cdot,\cdot)$ and the assumptions on $c^I(\cdot,\cdot)$ are straightforward to verify.

\end{proof}

%

\subsubsection{Matrix form}
The block preconditioner is as follows:

$$P_3^{III} = 
\left(
\begin{array}{ccc}
A_u &&\\
&A_v&\\
&& {\beta^{-1}}M_p+\kappa^{-1}B_vM_v^{-1}B_v^T+C_p\\
\end{array}
\right)^{-1},
$$
where $\beta = \max\{ \mu,\lambda\}$.

\subsection{Compare with Schur complement based preconditioners}

In \cite{Axelsson.O;Blaheta.R;Byczanski.P2012a}, block preconditioners are proposed for the discretized Biot model of the following form:
\begin{equation}
\label{eq:block_reg}
\left(
\begin{array}{ccc}
A_u & &B_u^T\\
& A_v& B_v^T\\
B_u&B_v&-C \\
\end{array}
\right)
\left(
\begin{array}{c}
\bu\\
\bv\\
p\\
\end{array}
\right)
=
\left(
\begin{array}{c}
f\\
r\\
g\\
\end{array}
\right),
\end{equation}
where $C$ is the pressure mass matrix $M_p$ with constant coefficient. Block preconditioners for the case $C=M_p$ are proposed:
\begin{itemize}
\item pressure Schur complement:
$$P_{ps}=\left(
\begin{array}{ccc}
A_u &&\\
&A_v&\\
&& -C-B_vD_v^{-1}B_v^T\\
\end{array}
\right)^{-1},$$
where $-C-B_vD_v^{-1}B_v^T$ is shown in \cite{Lipnikov.K2002a} to be spectrally equivalent to the exact Schur complement $-C-B_uA_u^{-1}B_u^T-B_vA_v^{-1}B_v^T$.  This preconditioner is also used in \cite{Turan.E;Arbenz.P2014a}.
\item displacement-velocity Schur complement:
$$P_{uvs}=\left(
\begin{array}{ccc}
A_u+\beta B_u^TC^{-1}B_u &\beta B_u^TC^{-1}B_v&\\
\beta B_v^TC^{-1}B_u&A_v+\beta B_v^TC^{-1}B_v&\\
&& -C\\
\end{array}
\right)^{-1}.
$$
\end{itemize}

{
It is shown in \cite{Axelsson.O;Blaheta.R;Byczanski.P2012a} that exact solutions of the first $2$-by-$2$ block of {$P_1^{III}$}, i.e.,
\begin{equation}
\label{eq:2by2block}
\left(
\begin{array}{cc}
A_u+\beta B_u^TM_p^{-1}B_u &\beta B_u^TM_p^{-1}B_v\\
\beta B_v^TM_p^{-1}B_u&A_v+\beta B_v^TM_p^{-1}B_v\\
\end{array}
\right)^{-1},
\end{equation}
result in uniform preconditioners. However, effective iterative solvers must be used for the inner iterations. In \cite{Axelsson.O;Blaheta.R;Byczanski.P2012a}, two block preconditioners, 
\begin{equation}
\label{eq:prec_2by2block}
\left(
\begin{array}{cc}
A_u &\\
&A_v\\
\end{array}
\right)^{-1}
\quad \mbox{ and }
\quad
\left(
\begin{array}{cc}
A_u+\beta B_u^TM_p^{-1}B_u &\\
&A_v+\beta B_v^TM_p^{-1}B_v\\
\end{array}
\right)^{-1},
\end{equation}
are used to precondition (\ref{eq:2by2block}).   The numerical tests in \cite{Axelsson.O;Blaheta.R;Byczanski.P2012a} show that the second preconditioner in (\ref{eq:prec_2by2block}) results in a far fewer iterations for the inner iterative solvers than the first one. However, this approach introduces an additional loop of iterative solvers.  In \cite{Ferronato.M;Castelletto.N;Gambolati.G2010a}, (\ref{eq:2by2block}) is directly approximated by the second preconditioner in  (\ref{eq:prec_2by2block}) and incomplete Cholesky factorization is used to further approximate the preconditioner.

The preconditioners, $P_2^{III}$ and $P_3^{III}$, that we proposed are provably optimal and given their block diagonal form they are easy to implement.  Further, $P_2^{III}$ and $P_3^{III}$ have another advantage: they apply to the case where the diagonal block matrix $ C=0$ (i.e., the fluid storage coefficient $S$ is zero), even though $P_2^{III}$ is subject to the constraint $\xi\leq \beta^{-1}$.
}

\subsection{Values of $\xi\beta$ for various  poroelastic materials}
Although some of the preconditioners we proposed depend on the assumption that $\xi\beta$ is uniformly bounded, $\xi\beta$ is usually small in various poroelastic materials. In Table \ref{tb:xibeta}, we calculate the corresponding values of $\xi\beta$ based on the poroelastic constants in \cite{CHENG.A2014a}.
\begin{table}[htp]
\caption{Values of $\xi\beta$ for various poroelastic materials}
\begin{center}
\begin{tabular}{|c|c||c|c|}
\hline
&$\xi\beta$&&$\xi\beta$\\
\hline
Ruhr sandstone &2.3836& Tennessee marble& 12.1667\\
Charcoal granite & 6.7635 & Berea sandstone& 2.3192 \\
Westerly granite & 2.5972 & Weber sandstone & 2.9235\\
Ohio sandstone & 3.5965 & Pecos sandstone & 2.5322\\
Boise sandstone & 2.4860&&\\
\hline
\end{tabular}
\end{center}
\label{tb:xibeta}
\end{table}%

\section{Numerical tests}
\label{sec:numerics}
We test the preconditioners using the poroelastic footing experiment (2D) (see \cite{Aguilar.G;Gaspar.F;Lisbona.F;Rodrigo.C2008a}).  The domain is $\Omega=(-4, 4)\times(-4, 4)$.  Define
$$\Gamma_1=\{(x,y)\in \partial\Omega, |x|\leq 0.8, y=4\},~~~\Gamma_2=\{(x,y)\in \partial\Omega, |x|> 0.8, y=4\}.$$

The boundary conditions are as follows:
\begin{equation*}
\begin{aligned}
&(\sigma_{e}- p\mathbf{I})\mathbf n=-10^4,  \mathbf v\cdot \mathbf n=0, &\mbox{on } \Gamma_1, \\
&(\sigma_{e}- p\mathbf{I})\mathbf n= 0, p=0, &\mbox{on } \Gamma_2, \\
&\mathbf{u}=0,\mathbf{v}\cdot \mathbf n=0, &\mbox{on } \partial\Omega\slash(\Gamma_1\cup\Gamma_2).\\
\end{aligned}
\end{equation*}
 We assume that the fluid storage coefficient is $S=0$ and the other material parameters are varying in huge range.
.
 
 We discretize the problem using FEniCS \cite{Logg.A;Mardal.K;Wells.G;others2012a}.  We show the robustness of the preconditioners with respect to problem sizes and varying parameters.  We discretize the problem on uniform triangular meshes.  We 
 use continuous Garlerkin (CG) method with $P_2\times RT_1\times P_0$  and  discontinuous Garlerkin (DG) method with $BDM_1 \times RT_1\times P_0$ for the three-field formulation \cite{hong2018parameter}. Thus, the inf-sup conditions of $b^{II}(\cdot,\cdot)$ for both $\mathbf{u}$-$p$ and $\mathbf{v}$-$p$ are satisfied.  We present the number of iterations of preconditioned MINRES (PMINRES) with the preconditioners for three-field formulation in Tables \ref{tableIII7}-\ref{tableIII0} (CG discretization) and Tables \ref{tableIII7DG}-\ref{tableIII0DG}  (DG discretization).  For each of the preconditioners we showed, the number of iterations does not vary much with respect to the changing parameters and problem sizes. 
 In addition, we also show the condition numbers of the unpreconditioned and preconditioned systems matrices on the coarsest mesh ($16\times 16$) in Table
 \ref{tb:condIII}.   The condition numbers of the preconditioned systems are almost constant and close to $1$.

\begin{table}[htbp]
\centering
\caption{Number of  iterations of PMINRES  with CG discretization for  three-field scheme with $\kappa = 10^{7}$.}
\vspace{1.5mm}
\begin{tabular}{|c|c||c|c|c||c|c|c||c|c|c|c|}\hline
  \multirow{2}{*}{ $\nu$}    &       \multirow{2}{*}{$h$}       
  & \multicolumn{3}{c||}{$E = 3\times 10^{4}$ }      &  \multicolumn{3}{c||}{$E = 3\times 10^{5}$}       &  \multicolumn{3}{c|}{$E = 3\times 10^{6} $}     \\
    \cline{3-11} 
&  &  $P_1^{III}$   &    $P_2^{III}$   &   $P_3^{III}$   
&    $P_1^{III}$   &    $P_2^{III}$   &   $P_3^{III}$    
&  $P_1^{III}$   &    $P_2^{III}$   &   $P_3^{III}$  \\  \hline
 \multirow{4}{*}{$0.4$ }   
 & $1/16$        & 28   & 35   & 27           & 36   & 51  & 43       & 34   & 49    & 51     \\ 
  \cline{2-11}
 & $1/32$        & 33   & 47   & 37           & 37   & 55  & 53       & 33   & 51    & 55    \\ 
  \cline{2-11}
 & $1/64$        & 35   & 54   & 47           & 36   & 57  & 57       & 33   & 51    & 57   \\
  \cline{2-11} 
 & $1/128$    & 36   & 56   & 55           & 36   & 58  & 60       & 33   & 51    & 59    \\    \hline\hline
 
 \multirow{4}{*}{$ 0.49$ }  
 & $1/16$          & 29   & 40   & 29           & 30   & 42  & 37       & 28   & 37    & 35     \\ 
  \cline{2-11}
 & $1/32$          & 30   & 42   & 35           & 30   & 42  & 39       & 28   & 37    & 37    \\ 
  \cline{2-11} 
 & $1/64$          & 30   & 44   & 39           & 30   & 42  & 39       & 26   & 37    & 37   \\
 \cline{2-11} 
 & $1/128$      & 29   & 44   & 39           & 29   & 42  & 41       & 26   & 37    & 37   \\   \hline\hline
  
 \multirow{4}{*}{$0.495$ }  
 & $1/16$          & 29   & 40   & 31           & 29   & 37  & 37       & 26   & 35    & 33     \\ 
  \cline{2-11}
 & $1/32$          & 29   & 42   & 35           & 28   & 38  & 37       & 26   & 35    & 35    \\ 
  \cline{2-11} 
 & $1/64$          & 28   & 42   & 37           & 28   & 38  & 39       & 26   & 35    & 35   \\ 
 \cline{2-11} 
 & $1/128$      & 28   & 42   & 39           & 28   & 38  & 39       & 26   & 35    & 35   \\     \hline \hline

\multirow{4}{*}{$0.499$ }  
 & $1/16$         & 27   & 39   & 33           & 26   & 35  & 33       & 23   & 30    & 31     \\ 
  \cline{2-11}
 & $1/32$         & 28   & 40   & 35           & 26   & 35  & 35       & 23   & 30    & 33     \\ 
  \cline{2-11} 
 & $1/644$         & 27   & 40   & 35           & 26   & 35  & 35       & 23   & 30    & 33     \\
 \cline{2-11} 
 & $1/128$     & 27   & 40   & 37           & 26   & 35  & 35       & 23   & 30    & 33    \\ \hline
\end{tabular}
\label{tableIII7}
\end{table}
\begin{table}[htbp]
\centering
\caption{Number of  iterations of PMINRES  with CG discretization for  three-field scheme with  $\kappa = 10^{5}$. }
\vspace{1.5mm}
\begin{tabular}{|c|c||c|c|c||c|c|c||c|c|c|c|}\hline
\multirow{2}{*}{ $\nu$}    &       \multirow{2}{*}{$h$}       
 & \multicolumn{3}{c||}{$E = 3\times 10^{4}$ }      &  \multicolumn{3}{c||}{$E = 3\times 10^{5}$}       &  \multicolumn{3}{c|}{$E = 3\times 10^{6} $}     \\
 \cline{3-11} 
&  
&  $P_1^{III}$   &    $P_2^{III}$   &   $P_3^{III}$   
&    $P_1^{III}$   &    $P_2^{III}$   &   $P_3^{III}$    
&  $P_1^{III}$   &    $P_2^{III}$   &   $P_3^{III}$  \\  \hline
\multirow{4}{*}{$0.4$ }  
 & $1/16$          & 34   & 49   & 51           & 29   & 40  & 41       & 24   & 33    & 35     \\ 
  \cline{2-11}
 & $1/32$          & 33   & 51   & 55           & 29   & 41  & 43       & 24   & 34    & 35    \\ 
  \cline{2-11} 
 & $1/64$          & 33   & 51   & 57           & 29   & 41  & 43       & 24   & 34    & 35   \\ 
 \cline{2-11} 
 & $1/128$      & 33   & 51   & 59           & 27   & 41  & 44       & 24   & 34    &37    \\ \hline\hline
 
 \multirow{4}{*}{$0.49$ }  
 & $1/16$          & 28   & 37   & 35           & 24   & 32  & 33       & 21   & 27    & 31     \\ 
  \cline{2-11}
 & $1/32$          & 28   & 37   & 37           & 24   & 32  & 35       & 21   & 28    & 33    \\ 
  \cline{2-11} 
 & $1/64$          & 26   & 37   & 37           & 24   & 32  & 35       & 21   & 28    & 33   \\ 
 \cline{2-11} 
 & $1/128$      & 26   & 37   & 37           & 24   & 32  & 35       & 21   & 28    & 33   \\   \hline\hline
  
 \multirow{4}{*}{$ 0.495$ }  
 & $1/16$          & 26   & 35   & 33           & 22   & 30  & 33       & 21   & 26    & 31     \\ 
  \cline{2-11}
 & $1/32$          & 26   & 35   & 35           & 22   & 31  & 33       & 21   & 27    & 31    \\ 
  \cline{2-11} 
 & $1/64$          & 26   & 35   & 35           & 22   & 31  & 33       & 21   & 27    & 31   \\ 
 \cline{2-11} 
 & $1/128$      & 26   & 35   & 35           & 22   & 31  & 33       & 21   & 27    & 31   \\   \hline\hline
 
  \multirow{4}{*}{$0.499$ }  
 & $1/16$          & 23   & 30   & 31           & 20   & 26  & 29       & 18   & 23    & 29     \\ 
  \cline{2-11}
 & $1/32$          & 23   & 30   & 33           & 20   & 26  & 31       & 19   & 23    & 29    \\ 
  \cline{2-11} 
 & $1/64$          & 23   & 30   & 33           & 20   & 26  & 31       & 19   & 23    & 29   \\ 
 \cline{2-11} 
 & $1/128$      & 23   & 30   & 33           & 20   & 26  & 31       & 19   & 23    & 29   \\   \hline
 
\end{tabular}
\label{tableIII5}
\end{table}

\begin{table}[htbp]
\centering
\caption{Number of  iterations of PMINRES  with CG discretization for  three-field scheme with $\kappa = 10^{3}$. }
\vspace{1.5mm}
\begin{tabular}{|c|c||c|c|c||c|c|c||c|c|c|c|}\hline
\multirow{2}{*}{ $\nu$}    &       \multirow{2}{*}{$h$}       
& \multicolumn{3}{c||}{$E = 3\times 10^{4}$ }      &  \multicolumn{3}{c||}{$E = 3\times 10^{5}$}       &  \multicolumn{3}{c|}{$E = 3\times 10^{6} $}     \\
\cline{3-11} 
&  
&  $P_1^{III}$   &    $P_2^{III}$   &   $P_3^{III}$   
&    $P_1^{III}$   &    $P_2^{III}$   &   $P_3^{III}$    
&  $P_1^{III}$   &    $P_2^{III}$   &   $P_3^{III}$  \\  \hline
 \multirow{4}{*}{$ 0.4$ }  
 & $1/16$          & 24   & 33   & 35           & 21   & 28  & 33       & 20   & 25    & 31     \\ 
  \cline{2-11}
 & $1/32$          & 24   & 34   & 35           & 21   & 28  & 33       & 20   & 25    & 31    \\ 
  \cline{2-11} 
 & $1/64$          & 24   & 34   & 35           & 21   & 28  & 33       & 20   & 25    & 31   \\ 
 \cline{2-11} 
 & $1/128$      & 24   & 34   & 37           & 21   & 28  & 33       & 20   & 25    & 31    \\  \hline\hline
 
 \multirow{4}{*}{$ 0.49$ }  
 & $1/16$          & 21   & 27   & 31           & 19   & 24  & 29       & 17   & 22    & 27     \\ 
  \cline{2-11}
 & $1/32$          & 21   & 28   & 33           & 19   & 24  & 31       & 17   & 22    & 29    \\ 
  \cline{2-11} 
 & $1/64$          & 21   & 28   & 33           & 19   & 24  & 31       & 17   & 22    & 29   \\ 
 \cline{2-11} 
 & $1/128$      & 21   & 28   & 33           & 19   & 24  & 31       & 17   & 22    & 29   \\   \hline\hline
  
 \multirow{4}{*}{$ 0.495$ }  
 & $1/16$          & 21   & 26   & 31           & 19   & 23  & 29       & 17   & 21    & 27     \\ 
  \cline{2-11}
 & $1/32$          & 21   & 27   & 31           & 19   & 23  & 29       & 17   & 21    & 27    \\ 
  \cline{2-11} 
 & $1/64$          & 21   & 27   & 31           & 19   & 23  & 29       & 17   & 21    & 27   \\
  \cline{2-11} 
 & $1/128$      & 21   & 27   & 31           & 19   & 23  & 29       & 17   & 21    & 27    \\ \hline \hline
 
  \multirow{4}{*}{$0.499$ }  
 & $1/16$          & 18   & 23   & 29           & 16   & 21  & 27       & 15   & 19    & 25     \\ 
  \cline{2-11}
 & $1/32$          & 19   & 23   & 29           & 16   & 21  & 27       & 15   & 19    & 25    \\ 
  \cline{2-11} 
 & $1/64$          & 19   & 23   & 29           & 16   & 21  & 27       & 15   & 19    & 25   \\
  \cline{2-11} 
 & $1/128$      & 19   & 23   & 29           & 16   & 21  & 27       & 15   & 19    & 25    \\  \hline
\end{tabular}
\label{tableIII3}
\end{table}
\begin{table}[htbp]
\centering
\caption{Number of  iterations of PMINRES  with CG discretization for  three-field scheme with  $\kappa = 10$. }
\vspace{1.5mm}
\begin{tabular}{|c|c||c|c|c||c|c|c||c|c|c|c|}\hline
\multirow{2}{*}{ $\nu$}    &       \multirow{2}{*}{$h$}       
& \multicolumn{3}{c||}{$E = 3\times 10^{4}$ }      &  \multicolumn{3}{c||}{$E = 3\times 10^{5}$}       &  \multicolumn{3}{c|}{$E = 3\times 10^{6} $}     \\
\cline{3-11} 
&  
 &  $P_1^{III}$   &    $P_2^{III}$   &   $P_3^{III}$   
&    $P_1^{III}$   &    $P_2^{III}$   &   $P_3^{III}$    
&  $P_1^{III}$   &    $P_2^{III}$   &   $P_3^{III}$  \\   \hline
 \multirow{4}{*}{$0.4$ }  
 & $1/16$          & 20   & 25   & 31           & 18   & 22  & 29       & 17   & 20    & 27     \\ 
  \cline{2-11}
 & $1/32$          & 20   & 25   & 31           & 18   & 22  & 29       & 17   & 20    & 27    \\ 
  \cline{2-11} 
 & $1/64$          & 20   & 25   & 31           & 18   & 22  & 29       & 17   & 20    & 27   \\ 
 \cline{2-11} 
 & $1/128$      & 20   & 25   & 31           & 18   & 22  & 29       & 18   & 20    & 27  \\    \hline\hline
 
 \multirow{4}{*}{$0.49$ }  
 & $1/16$          & 17   & 22   & 27           & 16   & 19  & 27       & 14   & 17    & 25     \\ 
  \cline{2-11}
 & $1/32$          & 17   & 22   & 29           & 16   & 20  & 27       & 14   & 18    & 25    \\ 
  \cline{2-11} 
 & $1/64$          & 17   & 22   & 29           & 16   & 20  & 27       & 14   & 17    & 25   \\ 
  \cline{2-11} 
 & $1/128$      & 17   & 22   & 29           & 16   & 20  & 27       & 14   & 17    & 25   \\    \hline\hline
  
 \multirow{4}{*}{$0.495$ }  
 & $1/16$          & 17   & 21   & 27           & 16   & 19  & 25       & 14   & 17    & 23     \\ 
  \cline{2-11}
 & $1/32$          & 17   & 21   & 27           & 16   & 19  & 25       & 14   & 17    & 23    \\ 
  \cline{2-11} 
 & $1/64$          & 17   & 21   & 27           & 16   & 19  & 25       & 14   & 17    & 23   \\ 
 \cline{2-11} 
 & $1/128$      & 17   & 21   & 27           & 16   & 19  & 25       & 14   & 17    & 23   \\    \hline\hline

  \multirow{4}{*}{$0.499$ }  
 & $1/16$          & 15   & 19   & 25           & 13   & 17  & 23       & 13   & 16    & 21     \\ 
  \cline{2-11}
 & $1/32$          & 15   & 19   & 25           & 13   & 17  & 23       & 13   & 16    & 21    \\ 
  \cline{2-11} 
 & $1/64$          & 15   & 19   & 25           & 13   & 17  & 23       & 13   & 16    & 21   \\ 
 \cline{2-11} 
 & $1/128$      & 15   & 19   & 25           & 13   & 17  & 23       & 13   & 16    & 21   \\    \hline
\end{tabular}
\label{tableIII1}
\end{table}
\begin{table}[htbp]
\centering
\caption{Number of  iterations of PMINRES  with CG discretization for  three-field scheme with $\kappa = 1$. }
\vspace{1.5mm}
\begin{tabular}{|c|c||c|c|c||c|c|c||c|c|c|c|}\hline
\multirow{2}{*}{ $\nu$}    &       \multirow{2}{*}{$h$}       
& \multicolumn{3}{c||}{$E = 3\times 10^{4}$ }      &  \multicolumn{3}{c||}{$E = 3\times 10^{5}$}       &  \multicolumn{3}{c|}{$E = 3\times 10^{6} $}     \\
\cline{3-11} 
&  
&  $P_1^{III}$   &    $P_2^{III}$   &   $P_3^{III}$   
&    $P_1^{III}$   &    $P_2^{III}$   &   $P_3^{III}$    
&  $P_1^{III}$   &    $P_2^{III}$   &   $P_3^{III}$  \\  \hline
 \multirow{4}{*}{$0.4$ }  
 & $1/16$          & 18   & 22   & 29           & 17   & 20  & 27       & 15   & 20    & 25     \\ 
  \cline{2-11}
 & $1/32$          & 18   & 22   & 29           & 17   & 20  & 27       & 15   & 20    & 25    \\ 
  \cline{2-11} 
 & $1/64$          & 18   & 22   & 29           & 17   & 20  & 27       & 15   & 20    & 25   \\
  \cline{2-11} 
 & $1/128$      & 18   & 22   & 29           & 17   & 20  & 27       & 15   & 20    & 25   \\    \hline\hline
 
 \multirow{4}{*}{$0.49$ }  
 & $1/16$          & 16   & 20   & 27           & 14   & 17  & 25       & 13   & 17    & 23     \\ 
  \cline{2-11}
 & $1/32$          & 16   & 20   & 27           & 14   & 18  & 25       & 13   & 17    & 23    \\ 
  \cline{2-11} 
 & $1/64$          & 16   & 20   & 27           & 14   & 17  & 25       & 13   & 17    & 23   \\ 
 \cline{2-11} 
 & $1/128$      & 16   & 20   & 27           & 14   & 17  & 25       & 13   & 17    & 23   \\  \hline\hline
  
 \multirow{4}{*}{$ 0.495$ }  
 & $1/16$          & 16   & 19   & 25           & 14   & 17  & 23       & 13   & 16    & 21     \\ 
  \cline{2-11}
 & $1/32$          & 16   & 19   & 25           & 14   & 17  & 23       & 13   & 16    & 21    \\ 
  \cline{2-11} 
 & $1/64$          & 16   & 19   & 25           & 14   & 17  & 23       & 13   & 16    & 21   \\ 
 \cline{2-11} 
 & $1/128$      & 16   & 19   & 25           & 14   & 17  & 23       & 13   & 16    & 21   \\   \hline\hline

 \multirow{4}{*}{$ 0.499$ }  
 & $1/16$          & 13   & 17   & 23           & 13   & 16  & 21       & 12   & 14    & 19     \\ 
  \cline{2-11}
 & $1/32$          & 13   & 17   & 23           & 13   & 16  & 21       & 12   & 14    & 19    \\ 
  \cline{2-11} 
 & $1/64$          & 13   & 17   & 23           & 13   & 16  & 21       & 12   & 14    & 19   \\
  \cline{2-11} 
 & $1/128$      & 13   & 17   & 23           & 13   & 16  & 21       & 12   & 14    & 19   \\    \hline
\end{tabular}
\label{tableIII0}
\end{table}

\begin{table}[htbp]
\centering
\caption{ Number of  iterations of PMINRES  with DG discretization for  three-field scheme with $\kappa = 10^{7}$.  }
\vspace{1.5mm}
\begin{tabular}{|c|c||c|c|c||c|c|c||c|c|c|c|}\hline
\multirow{2}{*}{ $\nu$}    &       \multirow{2}{*}{$h$}       
& \multicolumn{3}{c||}{$E = 3\times 10^{4}$ }      &  \multicolumn{3}{c||}{$E = 3\times 10^{5}$}       &  \multicolumn{3}{c|}{$E = 3\times 10^{6} $}     \\
\cline{3-11} 
&  
&  $P_1^{III}$   &    $P_2^{III}$   &   $P_3^{III}$   
&    $P_1^{III}$   &    $P_2^{III}$   &   $P_3^{III}$    
&  $P_1^{III}$   &    $P_2^{III}$   &   $P_3^{III}$  \\  
\hline 
 \multirow{4}{*}{$0.4$ }  
 & $1/16$        & 12   & 34   & 27           & 16   & 49  & 43       & 15   & 49    & 51     \\ 
  \cline{2-11}
 & $1/32$        & 14   & 46   & 37           & 16   & 55  & 53       & 15   & 51    & 51    \\ 
  \cline{2-11}
 & $1/64$        & 16   & 54   & 49           & 16   & 57  & 55       & 15   & 51    & 53   \\
  \cline{2-11} 
 & $1/128$    & 16   & 58   & 55           & 16   & 58  & 57       & 15   & 51    & 57    \\    \hline\hline
 
 \multirow{4}{*}{$0.49$ }  
 & $1/16$          & 10   & 40   & 29           & 10    & 41  & 37       & 16   & 37    & 35     \\ 
  \cline{2-11}
 & $1/32$          & 11   & 42   & 35           & 10   & 42   & 39       & 16   & 37    & 37    \\ 
  \cline{2-11}  
 & $1/64$          & 12   & 44   & 39           & 10   & 42   & 41       & 16   & 37    & 37   \\
 \cline{2-11} 
 & $1/128$      & 12   & 44   & 39           & 12   & 42   & 41       & 16   & 37    & 37   \\   \hline\hline
  
 \multirow{4}{*}{$0.495$ }  
 & $1/16$          & 10   & 40   & 31           & 10   & 37  & 37       & 16   & 35    & 35     \\ 
  \cline{2-11}
 & $1/32$          & 10   & 42   & 35           & 10   & 37  & 37       & 16   & 35    & 35    \\ 
  \cline{2-11} 
 & $1/64$          & 12   & 42   & 37           & 10   & 38  & 39       & 16   & 35    & 35   \\ 
 \cline{2-11} 
 & $1/128$      & 12   & 42   & 39           & 11   & 38  & 39       & 16   & 35    & 35   \\     \hline \hline

\multirow{4}{*}{$ 0.499$ }  
 & $1/16$         & 9     & 38   & 33           & 15   & 35  & 33       & 17   & 30    & 31     \\ 
  \cline{2-11}
 & $1/32$         & 10   & 39   & 35           & 15   & 35  & 35       & 17   & 30    & 33     \\ 
  \cline{2-11} 
 & $1/64$         & 11   & 40   & 37           & 15   & 35  & 35       & 17   & 30    & 33     \\
 \cline{2-11} 
 & $1/128$     & 11   & 40   & 37           & 15   & 35  & 35       & 17   & 30    & 33    \\ \hline
\end{tabular}
\label{tableIII7DG}
\end{table}

\begin{table}[htbp]
\centering
\caption{Number of  iterations of PMINRES  with DG discretization for  three-field scheme with $\kappa = 10^{5}$. }
\vspace{1.5mm}
\begin{tabular}{|c|c||c|c|c||c|c|c||c|c|c|c|}\hline
\multirow{2}{*}{ $\nu$}    &       \multirow{2}{*}{$h$}       
& \multicolumn{3}{c||}{$E = 3\times 10^{4}$ }      &  \multicolumn{3}{c||}{$E = 3\times 10^{5}$}       &  \multicolumn{3}{c|}{$E = 3\times 10^{6} $}     \\
\cline{3-11} 
&  
&  $P_1^{III}$   &    $P_2^{III}$   &   $P_3^{III}$   
&    $P_1^{III}$   &    $P_2^{III}$   &   $P_3^{III}$    
&  $P_1^{III}$   &    $P_2^{III}$   &   $P_3^{III}$  \\  \hline
 \multirow{4}{*}{$0.4$ }  
 & $1/16$          & 15   & 49   & 51           & 17   & 39  & 41       & 18   & 32    & 35     \\ 
  \cline{2-11}
 & $1/32$          & 15   & 51   & 51           & 17   & 39  & 43       & 18   & 33    & 37    \\ 
  \cline{2-11} 
 & $1/64$          & 15   & 51   & 53           & 17   & 40  & 43       & 18   & 33    & 37   \\ 
 \cline{2-11} 
 & $1/128$      & 15   & 51   & 59           & 17   & 41  & 45       & 18   & 33    &37    \\ \hline\hline
 
 \multirow{4}{*}{$0.49$ }  
 & $1/16$          & 16   & 37   & 35           & 17   & 32  & 33       & 17   & 27    & 31     \\ 
  \cline{2-11}
 & $1/32$          & 16   & 37   & 37           & 17   & 32  & 35       & 17   & 28    & 33    \\ 
  \cline{2-11} 
 & $1/64$          & 16   & 37   & 37           & 17   & 32  & 35       & 17   & 28    & 33   \\ 
 \cline{2-11} 
 & $1/128$      & 16   & 37   & 37           & 17   & 32  & 35       & 17   & 28    & 33   \\   \hline\hline
  
 \multirow{4}{*}{$0.495$ }  
 & $1/16$          & 16   & 35   & 35           & 17   & 29  & 33       & 17   & 26    & 31     \\ 
  \cline{2-11}
 & $1/32$          & 16   & 35   & 35           & 17   & 31  & 33       & 17   & 27    & 31    \\ 
  \cline{2-11} 
 & $1/64$          & 16   & 35   & 35           & 17   & 31  & 33       & 17   & 27    & 31   \\ 
 \cline{2-11} 
 & $1/128$      & 16   & 35   & 35           & 17   & 31  & 33       & 17   & 27    & 31   \\   \hline\hline
 
  \multirow{4}{*}{$ 0.499$ }  
 & $1/16$          & 17   & 30   & 31           & 16   & 26  & 31       & 16   & 23    & 29     \\ 
  \cline{2-11}
 & $1/32$          & 17   & 30   & 33           & 17   & 26  & 31       & 16   & 23    & 29    \\ 
  \cline{2-11} 
 & $1/64$          & 17   & 30   & 33           & 17   & 26  & 31       & 16   & 23    & 29   \\ 
 \cline{2-11} 
 & $1/128$      & 17   & 30   & 33           & 17   & 26  & 31       & 16   & 23    & 29   \\   \hline  \hline
 
\end{tabular}
\label{tableIII5DG}
\end{table}
\begin{table}[htbp]
\centering
\caption{Number of  iterations of PMINRES  with DG discretization for  three-field scheme with  $\kappa = 10^{3}$. }
\vspace{1.5mm}
\begin{tabular}{|c|c||c|c|c||c|c|c||c|c|c|c|}\hline
\multirow{2}{*}{ $\nu$}    &       \multirow{2}{*}{$h$}       
& \multicolumn{3}{c||}{$E = 3\times 10^{4}$ }      &  \multicolumn{3}{c||}{$E = 3\times 10^{5}$}       &  \multicolumn{3}{c|}{$E = 3\times 10^{6} $}     \\
\cline{3-11} 
&  
&  $P_1^{III}$   &    $P_2^{III}$   &   $P_3^{III}$   
&    $P_1^{III}$   &    $P_2^{III}$   &   $P_3^{III}$    
&  $P_1^{III}$   &    $P_2^{III}$   &   $P_3^{III}$  \\  \hline
 \multirow{4}{*}{$0.4$ }  
 & $1/16$          & 18   & 32   & 35           & 17   & 28  & 33       & 17   & 25    & 31     \\ 
  \cline{2-11}
 & $1/32$          & 18   & 33   & 37           & 18   & 28  & 33       & 17   & 25    & 31    \\ 
  \cline{2-11} 
 & $1/64$          & 18   & 33   & 37           & 18   & 29  & 33       & 17   & 25    & 31   \\ 
 \cline{2-11} 
 & $1/128$      & 18   & 33   & 37           & 18   & 29  & 33       & 17   & 25    & 31    \\  \hline\hline
 
 \multirow{4}{*}{$0.49$ }  
 & $1/16$          & 17   & 27   & 31           & 16   & 24  & 31       & 16   & 22    & 29     \\ 
  \cline{2-11}
 & $1/32$          & 17   & 28   & 33           & 17   & 24  & 31       & 16   & 22    & 29    \\ 
  \cline{2-11} 
 & $1/64$          & 17   & 28   & 33           & 17   & 24  & 31       & 16   & 22    & 29   \\ 
 \cline{2-11} 
 & $1/128$      & 17   & 28   & 33           & 17   & 24  & 31       & 16   & 22    & 29   \\   \hline\hline
  
 \multirow{4}{*}{$0.495$ }  
 & $1/16$          & 17   & 26   & 31           & 16   & 23  & 29       & 15   & 21    & 27     \\ 
  \cline{2-11}
 & $1/32$          & 17   & 27   & 31           & 16   & 23  & 29       & 15   & 21    & 27    \\ 
  \cline{2-11} 
 & $1/64$          & 17   & 27   & 31           & 16   & 23  & 29       & 15   & 21    & 27   \\
  \cline{2-11} 
 & $1/128$      & 17   & 27   & 31           & 16   & 23  & 29       & 15   & 21    & 27    \\ \hline \hline
 
  \multirow{4}{*}{$0.499$ }  
 & $1/16$          & 16   & 23   & 29           & 15   & 21  & 27       & 14   & 19    & 25     \\ 
  \cline{2-11}
 & $1/32$          & 16   & 23   & 29           & 15   & 21  & 27       & 14   & 19    & 25    \\ 
  \cline{2-11} 
 & $1/64$          & 16   & 23   & 29           & 15   & 21  & 27       & 14   & 19    & 25   \\
  \cline{2-11} 
 & $1/128$      & 16   & 23   & 29           & 15   & 21  & 27       & 14   & 19    & 25    \\  \hline
 
\end{tabular}
\label{tableIII3DG}
\end{table}

\begin{table}[htbp]
\centering
\caption{Number of  iterations of PMINRES  with DG discretization for  three-field scheme with  $\kappa = 10$. }
\vspace{1.5mm}
\begin{tabular}{|c|c||c|c|c||c|c|c||c|c|c|c|}\hline
\multirow{2}{*}{ $\nu$}    &       \multirow{2}{*}{$h$}       
& \multicolumn{3}{c||}{$E = 3\times 10^{4}$ }      &  \multicolumn{3}{c||}{$E = 3\times 10^{5}$}       &  \multicolumn{3}{c|}{$E = 3\times 10^{6} $}     \\
\cline{3-11} 
&  
&  $P_1^{III}$   &    $P_2^{III}$   &   $P_3^{III}$   
&    $P_1^{III}$   &    $P_2^{III}$   &   $P_3^{III}$    
&  $P_1^{III}$   &    $P_2^{III}$   &   $P_3^{III}$  \\  \hline
 \multirow{4}{*}{$0.4$ }  
 & $1/16$          & 17   & 25   & 31           & 16   & 22  & 29       & 15   & 20    & 27     \\ 
  \cline{2-11}
 & $1/32$          & 17   & 25   & 31           & 16   & 22  & 29       & 16   & 20    & 27    \\ 
  \cline{2-11} 
 & $1/64$          & 17   & 25   & 31           & 16   & 22  & 29       & 16   & 20    & 27   \\ 
 \cline{2-11} 
 & $1/128$      & 17   & 25   & 31           & 18   & 22  & 29       & 16   & 20    & 27  \\    \hline\hline
 
 \multirow{4}{*}{$0.49$ }  
 & $1/16$          & 16   & 22   & 29           & 15   & 20  & 27       & 14   & 17    & 25     \\ 
  \cline{2-11}
 & $1/32$          & 16   & 22   & 29           & 15   & 20  & 27       & 14   & 18    & 25    \\ 
  \cline{2-11} 
 & $1/64$          & 16   & 22   & 29           & 15   & 20  & 27       & 14   & 18    & 25   \\ 
  \cline{2-11} 
 & $1/128$      & 16   & 22   & 29           & 15   & 20  & 27       & 14   & 18    & 25   \\    \hline\hline
  
 \multirow{4}{*}{$0.495$ }  
 & $1/16$          & 15   & 21   & 27           & 15   & 19  & 25       & 14   & 17    & 23     \\ 
  \cline{2-11}
 & $1/32$          & 15   & 21   & 27           & 15   & 19  & 25       & 14   & 17    & 23    \\ 
  \cline{2-11} 
 & $1/64$          & 15   & 21   & 27           & 15   & 19  & 25       & 14   & 17    & 23   \\ 
 \cline{2-11} 
 & $1/128$      & 15   & 21   & 27           & 15   & 19  & 25       & 14   & 17    & 23   \\    \hline\hline

  \multirow{4}{*}{$0.499$ }  
 & $1/16$          & 14   & 19   & 25           & 13   & 16  & 23       & 12   & 16    & 21     \\ 
  \cline{2-11}
 & $1/32$          & 14   & 19   & 25           & 13   & 17  & 23       & 13   & 16    & 21    \\ 
  \cline{2-11} 
 & $1/64$          & 14   & 19   & 25           & 13   & 17  & 23       & 13   & 16    & 21   \\ 
 \cline{2-11} 
 & $1/128$      & 14   & 19   & 25           & 13   & 17  & 23       & 13   & 16    & 21   \\    \hline
 
\end{tabular}
\label{tableIII1DG}
\end{table}
\begin{table}[htbp]
\centering
\caption{Number of  iterations of PMINRES  with DG discretization for  three-field scheme with  $\kappa = 1$. }
\vspace{1.5mm}
\begin{tabular}{|c|c||c|c|c||c|c|c||c|c|c|c|}\hline
\multirow{2}{*}{ $\nu$}    &       \multirow{2}{*}{$h$}       
& \multicolumn{3}{c||}{$E = 3\times 10^{4}$ }      &  \multicolumn{3}{c||}{$E = 3\times 10^{5}$}       &  \multicolumn{3}{c|}{$E = 3\times 10^{6} $}     \\
\cline{3-11} 
&  
&  $P_1^{III}$   &    $P_2^{III}$   &   $P_3^{III}$   
&    $P_1^{III}$   &    $P_2^{III}$   &   $P_3^{III}$    
&  $P_1^{III}$   &    $P_2^{III}$   &   $P_3^{III}$  \\  \hline
 \multirow{4}{*}{$ 0.4$ }  
  & $1/16$          & 16   & 22   & 29           & 15   & 20  & 27       & 15   & 19    & 25     \\ 
  \cline{2-11}
 & $1/32$          & 16   & 22   & 29           & 16   & 20  & 27       & 15   & 19    & 25    \\ 
  \cline{2-11} 
 & $1/64$          & 16   & 22   & 29           & 16   & 20  & 27       & 15   & 20    & 25   \\
  \cline{2-11} 
 & $1/128$      & 16   & 22   & 29           & 16   & 20  & 27       & 15   & 20    & 25   \\    \hline\hline
 
 \multirow{4}{*}{$ 0.49$ }  
 & $1/16$          & 15   & 20   & 27           & 14   & 17  & 25       & 13   & 16    & 23     \\ 
  \cline{2-11}
 & $1/32$          & 15   & 20   & 27           & 14   & 18  & 25       & 13   & 17    & 23    \\ 
  \cline{2-11} 
 & $1/64$          & 15   & 20   & 27           & 14   & 18  & 25       & 13   & 17    & 23   \\ 
 \cline{2-11} 
 & $1/128$      & 15   & 20   & 27           & 14   & 18  & 25       & 13   & 17    & 23   \\  \hline\hline
  
 \multirow{4}{*}{$0.495$ }  
 & $1/16$          & 15   & 19   & 25           & 14   & 17  & 23       & 13   & 16    & 21     \\ 
  \cline{2-11}
 & $1/32$          & 15   & 19   & 25           & 14   & 17  & 23       & 13   & 16    & 23    \\ 
  \cline{2-11} 
 & $1/64$          & 15   & 19   & 25           & 14   & 17  & 23       & 13   & 16    & 23   \\ 
 \cline{2-11} 
 & $1/128$        & 15   & 19   & 25           & 14   & 17  & 23       & 13   & 16    & 23   \\   \hline\hline

 \multirow{4}{*}{$0.499$ }  
 & $1/16$          & 13   & 17   & 23           & 12   & 16  & 21       & 12   & 13    & 19     \\ 
  \cline{2-11}
 & $1/32$          & 13   & 17   & 23           & 13   & 16  & 21       & 12   & 14    & 19    \\ 
  \cline{2-11} 
 & $1/s64$        & 13   & 17   & 23           & 13   & 16  & 21       & 12   & 14    & 19   \\
  \cline{2-11} 
 & $1/128$        & 13   & 17   & 23           & 13   & 16  & 21       & 12   & 14    & 19   \\    \hline
 
\end{tabular}
\label{tableIII0DG}
\end{table}

\begin{table}[htp]
\caption{\small Conditioned number of the unpreconditioned and preconditioned (with $P_i^{III}$, $i=1,2,3$) system matrices on coarsest mesh}
\begin{center}
\begin{tabular}{|c|c|c|c|c|c|c|c|c|}
\hline 
	       &   \multicolumn{4}{c|}{$E=3\times 10^4$}  \\ \hline
  $\nu$ & N/A  & $P_1^{III}$ & $P_2^{III}$ &  $P_3^{III}$  \\
  $0.2$& $1.15\times 10^6$ & 1.27 & 3.67 & 4.33  \\ 
    $0.49$& $5.37\times 10^6$ & 1.05 & 4.00 & 4.29  \\
      $0.495$&  $1.08\times 10^7$ & 1.05 & 2.67 & 4.31  \\
                \hline	
  \end{tabular}
\end{center}
\begin{center}
\begin{tabular}{|c|c|c|c|c|c|c|c|c|}
\hline 
	       &  \multicolumn{4}{c|}{$E=3\times 10^5$} \\ \hline
  $\nu$ & N/A  & $P_1^{III}$ & $P_2^{III}$ &  $P_3^{III}$  \\
  $0.2$&  $1.15\times 10^7$ & 1.03 & 3.68 & 4.30 \\ 
    $0.49$&  $5.36\times 10^7$ & 1.01 & 4.00 & 4.30  \\
      $0.495$& $1.08\times 10^8$ & 1.01 & 2.67 & 4.30  \\
                \hline	
  \end{tabular}
\end{center}

\begin{center}
\begin{tabular}{|c|c|c|c|c|c|c|c|c|}
\hline 
	       &   \multicolumn{4}{c|}{$E=3\times 10^6$} \\ \hline
  $\nu$ & N/A  & $P_1^{III}$ & $P_2^{III}$ &  $P_3^{III}$  \\
  $0.2$&  $1.15\times 10^8$ & 1.00 & 3.68 & 4.30 \\ 
    $0.49$&  $5.36\times 10^8$ & 1.00 & 4.00 & 4.30 \\
      $0.495$&   $1.08\times 10^9$ & 1.00 & 2.67 & 4.30  \\
                \hline	
  \end{tabular}
\end{center}

\label{tb:condIII}
\end{table}%

\section{Concluding remarks}

In this paper we studied the well-posedness of the linear systems arising from discretized poroelasticity problems. We formulate block preconditioner for the two-filed Biot model and several preconditioners 
for the classical three-filed Biot model under the unified relationship framework between well-posedness and preconditioners.  By the unified theory, we show all the considered preconditioners are uniformly optimal with respect to material and discretization parameters.  
 The preconditioners have block diagonal form and reduce the global preconditioning to the local preconditioning. Numerical experiments have demonstrated the robustness of the preconditioners.  Although the blocks are inverted using direct method, we expect that replacing direct solvers by preconditioned iterative solvers (like multigrid preconditioned MINRES) will result in robust iterative solvers for the whole systems.  Although only block diagonal preconditioners are derived in this paper, these preconditioners can be also used to develop block triangular preconditioners \cite{Loghin.D;Wathen.A2004a}.  

\bibliographystyle{spmpsci}      
\bibliography{library}   

%
%

\end{document}

%% file: symbols.tex
\newcommand*{\ucong}{\mathrel{\text{\raisebox{.25ex}{\rotatebox[origin=c]{180}{$\cong$}}}}}

\def\ec{\mathrel{\hbox{$\copy\Ea\kern-\wd\Ea\raise-3.5pt\hbox{$\sim$}$}}}

\newcommand{\bx}{\mathbf{x}}

\newcommand{\bu}{\mathbf{u}}
\newcommand{\bv}{\mathbf{v}}

\newcommand{\bI}{\mathbf{I}}

\newcommand{\bn}{\mathbf{n}}

\newcommand{\beq}{\begin{equation}}
\newcommand{\eeq}{\end{equation}}
\newcommand{\beqa}{\begin{eqnarray}}
\newcommand{\eeqa}{\end{eqnarray}}

\newcommand{\cP}{{\mathcal P}}

\newcommand{\by}{{\bf y}}













%% file: main.bbl
\begin{thebibliography}{10}
\providecommand{\url}[1]{{#1}}
\providecommand{\urlprefix}{URL }
\expandafter\ifx\csname urlstyle\endcsname\relax
  \providecommand{\doi}[1]{DOI~\discretionary{}{}{}#1}\else
  \providecommand{\doi}{DOI~\discretionary{}{}{}\begingroup
  \urlstyle{rm}\Url}\fi

\bibitem{adler2019robust}
Adler, J.H., Gaspar, F.J., Hu, X., Ohm, P., Rodrigo, C., Zikatanov, L.T.:
  Robust preconditioners for a new stabilized discretization of the poroelastic
  equations.
\newblock arXiv preprint arXiv:1905.10353  (2019)

\bibitem{adler2017robust}
Adler, J.H., Gaspar, F.J., Hu, X., Rodrigo, C., Zikatanov, L.T.: Robust block
  preconditioners for biot?s model.
\newblock In: International Conference on Domain Decomposition Methods, pp.
  3--16. Springer (2017)

\bibitem{Aguilar.G;Gaspar.F;Lisbona.F;Rodrigo.C2008a}
Aguilar, G., Gaspar, F., Lisbona, F., Rodrigo, C.: Numerical stabilization of
  biot's consolidation model by a perturbation on the flow equation.
\newblock International Journal for Numerical Methods in Engineering
  \textbf{75}(11), 1282--1300 (2008)

\bibitem{Axelsson.O;Blaheta.R;Byczanski.P2012a}
Axelsson, O., Blaheta, R., Byczanski, P.: Stable discretization of
  poroelasticity problems and efficient preconditioners for arising saddle
  point type matrices.
\newblock Computing and Visualization in Science \textbf{15}(4), 191--207
  (2012)

\bibitem{Babuska.I1971a}
Babu{\v{s}}ka, I.: Error-bounds for finite element method.
\newblock Numerische Mathematik \textbf{16}(4), 322--333 (1971)

\bibitem{Benzi.M;Olshanskii.M2006a}
Benzi, M., Olshanskii, M.A.: {An augmented Lagrangian-based approach to the
  Oseen problem}.
\newblock SIAM Journal on Scientific Computing \textbf{28}(6), 2095--2113
  (2006)

\bibitem{Bergamaschi.L;Ferronato.M;Gambolati.G2007a}
Bergamaschi, L., Ferronato, M., Gambolati, G.: Novel preconditioners for the
  iterative solution to fe-discretized coupled consolidation equations.
\newblock Computer Methods in Applied Mechanics and Engineering
  \textbf{196}(25), 2647--2656 (2007)

\bibitem{Biot.M1941a}
Biot, M.A.: General theory of three-dimensional consolidation.
\newblock Journal of Applied Physics \textbf{12}(2), 155--164 (1941)

\bibitem{Biot.M1955a}
Biot, M.A.: Theory of elasticity and consolidation for a porous anisotropic
  solid.
\newblock Journal of Applied Physics \textbf{26}(2), 182--185 (1955)

\bibitem{Boffi.D;Brezzi.F;Fortin.M2013a}
Boffi, D., Brezzi, F., Fortin, M.: Mixed finite element methods and
  applications, \emph{Springer Series in Computational Mathematics}, vol.~44.
\newblock Springer Berlin Heidelberg (2013)

\bibitem{Bramble.J2003a}
Bramble, J.H.: A proof of the inf--sup condition for the stokes equations on
  lipschitz domains.
\newblock Mathematical Models and Methods in Applied Sciences \textbf{13}(03),
  361--371 (2003)

\bibitem{Bramble.J;Lazarov.R;Pasciak.J2001a}
Bramble, J.H., Lazarov, R.D., Pasciak, J.E.: Least-squares methods for linear
  elasticity based on a discrete minus one inner product.
\newblock Computer methods in applied mechanics and engineering
  \textbf{191}(8), 727--744 (2001)

\bibitem{Brezzi.F;Fortin.M1991a}
Brezzi, F., Fortin, M.: Mixed and hybrid finite element methods.
\newblock Springer-Verlag New York (1991)

\bibitem{CHENG.A2014a}
Cheng, A.H.D.: Fundamentals of poroelasticity.
\newblock Analysis and Design Methods: Comprehensive Rock Engineering:
  Principles, Practice and Projects p. 113 (2014)

\bibitem{Elman.H;Silvester.D;Wathen.A2005a}
Elman, H., Silvester, D., Wathen, A.: {Finite elements and fast iterative
  solvers: with applications in incompressible fluid dynamics}.
\newblock Oxford University Press, USA (2005)

\bibitem{Ferronato.M;Castelletto.N;Gambolati.G2010a}
Ferronato, M., Castelletto, N., Gambolati, G.: A fully coupled 3-d mixed finite
  element model of biot consolidation.
\newblock Journal of Computational Physics \textbf{229}(12), 4813--4830 (2010)

\bibitem{Ferronato.M;Gambolati.G;Teatini.P2001a}
Ferronato, M., Gambolati, G., Teatini, P.: Ill-conditioning of finite element
  poroelasticity equations.
\newblock International Journal of Solids and Structures \textbf{38}(34),
  5995--6014 (2001)

\bibitem{Hackbusch.W1985a}
Hackbusch, W.: Multi-grid methods and applications, vol.~4.
\newblock Springer-Verlag Berlin (1985)

\bibitem{Haga.J;Osnes.H;Langtangen.H2011a}
Haga, J.B., Osnes, H., Langtangen, H.P.: Efficient block preconditioners for
  the coupled equations of pressure and deformation in highly discontinuous
  media.
\newblock International Journal for Numerical and Analytical Methods in
  Geomechanics \textbf{35}(13), 1466--1482 (2011)

\bibitem{Haga.J;Osnes.H;Langtangen.H2012b}
Haga, J.B., Osnes, H., Langtangen, H.P.: On the causes of pressure oscillations
  in low-permeable and low-compressible porous media.
\newblock International Journal for Numerical and Analytical Methods in
  Geomechanics \textbf{36}(12), 1507--1522 (2012)

\bibitem{Haga.J;Osnes.H;Langtangen.H2012a}
Haga, J.B., Osnes, H., Langtangen, H.P.: A parallel block preconditioner for
  large-scale poroelasticity with highly heterogeneous material parameters.
\newblock Computational Geosciences \textbf{16}(3), 723--734 (2012)

\bibitem{Hiptmair.R;Xu.J2007a}
Hiptmair, R., Xu, J.: {Nodal auxiliary space preconditioning in ${\bf H}({\rm
  curl})$ and ${\bf H}({\rm div})$ spaces}.
\newblock SIAM J. Numer. Anal. \textbf{45}(6), 2483----2509 (electronic)
  (2007).
\newblock \doi{10.1137/060660588}.
\newblock \urlprefix\url{http://dx.doi.org/10.1137/060660588}

\bibitem{hong2018parameter}
Hong, Q., Kraus, J.: Parameter-robust stability of classical three-field
  formulation of biot's consolidation model.
\newblock Electronic Transactions on Numerical Analysis \textbf{48}, 202--227
  (2018)

\bibitem{hong2018conservative}
Hong, Q., Kraus, J., Lymbery, M., Philo, F.: Conservative discretizations and
  parameter-robust preconditioners for biot and multiple-network flux-based
  poroelasticity models.
\newblock Numerical Linear Algebra with Applications p. e2242 (2018)

\bibitem{hong2019parameter}
Hong, Q., Kraus, J., Lymbery, M., Philo, F.: Parameter-robust uzawa-type
  iterative methods for double saddle point problems arising in biot's
  consolidation and multiple-network poroelasticity models.
\newblock arXiv preprint arXiv:1910.05883  (2019)

\bibitem{hong2018fixed}
Hong, Q., Kraus, J., Lymbery, M., Wheeler, M.F.: Parameter-robust convergence
  analysis of fixed-stress split iterative method for multiple-permeability
  poroelasticity systems.
\newblock arXiv preprint arXiv:1812.11809  (2018)

\bibitem{Lee.J;Mardal.K;Winther.R2015a}
Lee, J.J., Mardal, K.A., Winther, R.: Parameter-robust discretization and
  preconditioning of biot's consolidation model.
\newblock arXiv preprint arXiv:1507.03199  (2015)

\bibitem{lee2017parameter}
Lee, J.J., Mardal, K.A., Winther, R.: Parameter-robust discretization and
  preconditioning of biot's consolidation model.
\newblock SIAM Journal on Scientific Computing \textbf{39}(1), A1--A24 (2017)

\bibitem{Lipnikov.K2002a}
Lipnikov, K.: Numerical methods for the biot model in poroelasticity.
\newblock Ph.D. thesis, University of Houston (2002)

\bibitem{Logg.A;Mardal.K;Wells.G;others2012a}
Logg, A., Mardal, K.A., Wells, G.N., et~al.: Automated solution of differential
  equations by the finite element method, \emph{Lecture Notes in Computational
  Science and Engineering}, vol.~84.
\newblock Springer Berlin Heidelberg (2012)

\bibitem{Loghin.D;Wathen.A2004a}
Loghin, D., Wathen, A.J.: Analysis of preconditioners for saddle-point
  problems.
\newblock SIAM Journal on Scientific Computing \textbf{25}(6), 2029--2049
  (2004)

\bibitem{Mardal.K;Winther.R2011a}
Mardal, K.A., Winther, R.: Preconditioning discretizations of systems of
  partial differential equations.
\newblock Numerical Linear Algebra with Applications \textbf{18}(1), 1--40
  (2011)

\bibitem{Murad.M;Loula.A1994a}
Murad, M.A., Loula, A.F.: On stability and convergence of finite element
  approximations of biot's consolidation problem.
\newblock International Journal for Numerical Methods in Engineering
  \textbf{37}(4), 645--667 (1994)

\bibitem{Phoon.K;Toh.K;Chan.S;Lee.F2002a}
Phoon, K., Toh, K., Chan, S., Lee, F.: An efficient diagonal preconditioner for
  finite element solution of biot's consolidation equations.
\newblock International Journal for Numerical Methods in Engineering
  \textbf{55}(4), 377--400 (2002)

\bibitem{Raviart.P;Thomas.J1977a}
Raviart, P.A., Thomas, J.M.: A mixed finite element method for 2-nd order
  elliptic problems.
\newblock In: Mathematical aspects of finite element methods, pp. 292--315.
  Springer (1977)

\bibitem{Rice.J2001a}
Rice, J.: Elasticity of fluid-infiltrated porous solids.
\newblock Handout of Advanced Environmental Geomechanics  (2001)

\bibitem{rodrigo2018new}
Rodrigo, C., Hu, X., Ohm, P., Adler, J.H., Gaspar, F.J., Zikatanov, L.T.: New
  stabilized discretizations for poroelasticity and the stokes? equations.
\newblock Computer Methods in Applied Mechanics and Engineering \textbf{341},
  467--484 (2018)

\bibitem{Toh.K;Phoon.K;Chan.S2004a}
Toh, K.C., Phoon, K.K., Chan, S.H.: Block preconditioners for symmetric
  indefinite linear systems.
\newblock International Journal for Numerical Methods in Engineering
  \textbf{60}(8), 1361--1381 (2004)

\bibitem{Turan.E;Arbenz.P2014a}
Turan, E., Arbenz, P.: Large scale micro finite element analysis of 3d bone
  poroelasticity.
\newblock Parallel Computing \textbf{40}(7), 239 -- 250 (2014).
\newblock \doi{http://dx.doi.org/10.1016/j.parco.2013.09.002}.
\newblock
  \urlprefix\url{http://www.sciencedirect.com/science/article/pii/S0167819113001063}.
\newblock 7th Workshop on Parallel Matrix Algorithms and Applications

\bibitem{Xu.J1992a}
Xu, J.: Iterative methods by space decomposition and subspace correction.
\newblock SIAM review \textbf{34}(4), 581--613 (1992)

\bibitem{Xu.J2015a}
Xu, J.: Finite element methods (2015).
\newblock Lecture notes

\bibitem{Xu.J;Yang.K2015a}
Xu, J., Yang, K.: Well-posedness and robust preconditioners for discretized
  fluid--structure interaction systems.
\newblock Computer Methods in Applied Mechanics and Engineering
  \textbf{292}(0), 69--91 (2015).
\newblock \doi{http://dx.doi.org/10.1016/j.cma.2014.09.034}.
\newblock {S}pecial Issue on Advances in Simulations of Subsurface Flow and
  Transport (Honoring Professor Mary F. Wheeler)

\end{thebibliography}
